\documentclass[reqno]{amsart}
\usepackage[english]{babel}
\usepackage{amscd,amssymb,amsmath,amsfonts,latexsym,amsthm}
\usepackage{inputenc}
\usepackage{graphicx,color}
\usepackage{tikz-cd}
\newtheorem{thm}{Theorem}[section]

\newtheorem{lem}[thm]{Lemma}
\newtheorem{prop}[thm]{Proposition}
\newtheorem{defn}[thm]{Definition}
\newtheorem{exam}[thm]{Example}
\newtheorem{rem}[thm]{Remark}
\numberwithin{equation}{section}

\def\cal{\mathcal }

\begin{document}

\title[Irreducible cuspidal modules of simple $n$-Lie algebras.]{Irreducible cuspidal modules of simple $n$-Lie algebras}

\author{Bakhrom Omirov, Gulkhayo Solijanova}

\address{Bakhrom A. Omirov \newline \indent
Institute for Advanced Study in Mathematics,
Harbin Institute of Technology, Harbin 150001 \newline \indent
Suzhou Research Institute, Harbin Institute of Technology, Harbin  215104, Suzhou, China}
\email{{\tt omirovb@mail.ru}}

\address{Gulkhayo O. Solijanova \newline \indent
National University of Uzbekistan, 100174, Tashkent, Uzbekistan,
\newline \indent
V.I. Romanovskiy Institute of Mathematics, Uzbekistan Academy of Sciences, Tashkent, Uzbekistan}\email{{\tt gulhayo.solijonova@mail.ru}}

\begin{abstract} This work devoted to the description of irreducible cuspidal modules over simple $n$-Lie algebras. Since the description of irreducible modules over $n$-Lie algebra $O^n$ are already well understood, we focus here on the irreducible cuspidal modules over $n$-Lie algebras of Wronskians and Jacobians. First, for a given $n$-Lie algebra $\mathcal{L}$, we analyze the possible Lie and Leibniz structures on $\wedge^{n-1} \mathcal{L}$ and $\otimes^{n-1} \mathcal{L}$ by thoroughly examining existing structures. Next, we classify the irreducible cuspidal modules over the $n$-Lie algebra of Wronskians defined on Laurent polynomials with degree-preserving derivations. Furthermore, we prove that these modules remain irreducible over the $n$-Lie algebra of Jacobians.
\end{abstract}



\maketitle

\section{Introduction}

\

The simplest phase space for Hamiltonian mechanics is $\mathbb{R}^2$ with coordinates $x$ and $y$, where the canonical Poisson bracket is given by
\[
\{f_1, f_2\} = \frac{\partial(f_1, f_2)}{\partial (x, y)},
\]
which satisfies the Jacobi identity and governs the Hamiltonian equations of motion. In \cite{Nambu1}, a generalization of this example was considered by defining, on $\mathbb{R}^3$ with coordinates $x, y, z$, the canonical bracket given as follows:
\begin{equation}\label{eqPoisson}
\{f_1, f_2, f_3\} = \frac{\partial(f_1, f_2, f_3)}{\partial(x, y, z)},
\end{equation}
where $\frac{\partial(f_1, f_2, f_3)}{\partial(x, y, z)}$ is the Jacobian of the map $f = (f_1, f_2, f_3) : \mathbb{R}^3 \to \mathbb{R}^3$. The formula \eqref{eqPoisson} naturally generalizes the Poisson bracket from a binary to a ternary operation.

The concept of $n$-Lie algebras was first introduced in the work of Filippov \cite{Filippov0}. He proposed this generalization of the ternary bracket to an $n$-bracket that satisfies a specific type of identity (see identity \eqref{identity1}), inspired by the general notion of $\Omega$-algebras introduced by Kurosh in \cite{Kurosh}.

In addition to simple finite-dimensional $n$-Lie algebras, conceptual examples of $n$-Lie algebras using determinants of Jacobians in the polynomial algebra in $n$ variables are provided in \cite{Filippov0, Filippov}. Later, the $n$ bracket, called the Nambu bracket of order $n$ - was independently studied in \cite{Takhtajan1} as a consistency condition for Nambu's dynamics. In this context, $n$-Lie algebras were referred to as ``Nambu-Lie algebras.''

The origin of Nambu-Lie algebras is related to Nambu mechanics, a generalization of Hamiltonian mechanics proposed by Y. Nambu in \cite{Nambu1}.

Recall that $n$-Lie algebras were introduced as a generalization of Lie algebras by replacing the traditional skew-symmetric binary operation with an $n$-ary operation that is skew-symmetric with respect to all its arguments. This $n$-ary operation, in the case of $n=2$, coincides with the Jacobi identity.

By appropriately defining the operators $ad$, the fundamental identity characterizing the $n$-Lie algebras can be interpreted as a type of derivation property for these algebras.

Another significant class of $3$-algebras, known as the {\it Lie triple systems}, was introduced in \cite{Lister}. These are vector spaces equipped with a bracket $[-, -, -]$ that satisfies the identity (see identity \eqref{identity1}) for $n=3$, and instead of skew-symmetry, the bracket satisfies the following conditions:
\[
[x, y, z] + [y, z, x] + [z, x, y] = 0 \quad \text{and} \quad [x, y, y] = 0.
\]

Lie triple systems later appeared independently in 1985 under the name ``skew-symmetric triple systems'' in the work \cite{Faulkner}, as part of a classification of identities in triple systems. Thus, $3$-Lie algebras were also referred to as skew-symmetric Lie triple systems.

It is worth mentioning the works \cite{Casas}, \cite{Takhtajan}, and \cite{Gautheron}, which further explored generalizations of Lie structures by extending the binary bracket to an $n$-bracket.

For a given $n$-Lie algebra (or Leibniz $n$-algebra) $\mathcal{L}$, various works have explored Lie and Leibniz structures on $\wedge^{n-1}\mathcal{L}$ and $\otimes^{n-1}\mathcal{L}$ (see \cite{erratum, Balibanu0, Balibanu, Casas, Takhtajan, Dzhum, Ladra, Takhtajan1}). However, these studies do not provide a unified approach to describe such structures. Therefore, we present the possible Lie and Leibniz structures on $\wedge^{n-1}\mathcal{L}$ and $\otimes^{n-1}\mathcal{L}$ by examining existing structures and addressing certain inaccuracies.

The study of irreducible representations of simple $n$-Lie algebras relies on a bijective correspondence between representations of simple $n$-Lie algebras and representations of Lie algebras of their inner derivations, subject to an additional property (see \cite{erratum, Balibanu, Boyallian, Boyallian1, Dzhum}). In fact, for an arbitrary $n$-Lie algebra $\mathcal{L}$, the space $\wedge^{n-1}\mathcal{L}$ can be equipped with a binary operation that makes it an algebra. Furthermore, the map
\[ \mathrm{ad} : \wedge^{n-1}\mathcal{L} \to \mathcal{L} \]
is such that $\ker(\mathrm{ad})$ is an ideal of the algebra $\wedge^{n-1}\mathcal{L}$. The Lie algebra of inner derivations of the $n$-Lie algebra $\mathcal{L}$ is then isomorphic to the quotient Lie algebra $\wedge^{n-1}\mathcal{L}/\ker(\mathrm{ad})$. This establishes the key relationship between the representations (modules) of the $n$-Lie algebra $\mathcal{L}$ and the Lie algebra $\wedge^{n-1}\mathcal{L}/\ker(\mathrm{ad})$.

W. Ling demonstrated in \cite{Ling} that there exists, up to isomorphism, only one simple finite-dimensional $n$ -Lie algebra for $n \geq 3$, denoted by $O^n$. The classification of irreducible finite-dimensional $O^n$-modules was obtained in \cite{Dzhum}, while all irreducible $O^n$-modules, including infinite-dimensional ones, were classified in \cite{Balibanu0}, \cite{Balibanu}.

Simple infinite-dimensional $n$-Lie algebras are known to exist only as the algebras of Wronskians and Jacobians, denoted by $W^n$ and $S^n$, respectively. Irreducible modules over the simple $n$-Lie algebras $W^n$ and $S^n$, defined on the ring of formal power series in $n$ variables, were described in \cite{Boyallian} and \cite{Boyallian1}, respectively.

It is known that the Lie algebras of inner derivations of $W^n$ and $S^n$ defined on Laurent polynomials with degree-preserving derivations are the Lie algebra of vector fields on a torus and its subalgebra of divergence-free vector fields, denoted as $W_n$ and $S_n$, respectively \cite{Kac}. Using the descriptions of irreducible cuspidal modules for $W_n$ and $S_n$ provided in \cite{Slava} and \cite{Talboom}, we identify those modules that can be extended to the simple $n$-Lie algebras of Wronskians and Jacobians defined on Laurent polynomials with
degree-preserving derivations.


We also note that the papers \cite{Billig}, \cite{Li}, and \cite{Xue} discuss the descriptions of irreducible Harish-Chandra modules, including cuspidal modules, for Lie algebras associated with $W_n$ and $S_n$.

The paper is organized as follows: In Section 3, we review the existing Lie and Leibniz structures on $\wedge^{n-1}\mathcal{L}$ and $\otimes^{n-1}\mathcal{L}$ and propose some additional structures. These structures are summarized in diagrams presented in Subsections 3.1 and 3.2.
Section 4 is devoted to establishing relations between modules over $n$-Lie algebras and modules over Lie algebras of their inner derivations. This investigation is motivated by inaccuracies in the works of \cite{Boyallian}, \cite{Boyallian1} concerning the relationship between $W^n$-modules and $S^n$-modules and $\wedge^{n-1} V$-modules. However, these inaccuracies do not affect the main results of those papers, as they stem from the fact that $\wedge^{n-1} V$ is a non-Lie Leibniz algebra and results regarding the ideal $\ker(\mathrm{ad})$ (see Remarks \ref{exam04} and \ref{counterexample}). Specifically, we establish the relation between modules over an $n$-Lie algebra $\mathcal{L}$ and the quotient Leibniz algebra $\wedge^{n-1}\mathcal{L} / \ker(\mathrm{ad})$, which is, in fact, a Lie algebra. A key role in this relationship is played by so-called the Balanced Equality (see Equality \eqref{eq3.7}). In particular, $M$ is an irreducible module of the $n$-Lie algebra $\mathcal{L}$ if and only if it is a $\wedge^{n-1}\mathcal{L} / \ker(\mathrm{ad})$-module that satisfies the Balanced Equality.
Section 5 provides descriptions of irreducible cuspidal $W_{n-1}$- and $S_n$-modules as obtained in \cite{Slava} and \cite{Talboom} (see Theorems \ref{mainthm} and \ref{Talboom}). To apply a unified approach in our study, these descriptions are reformulated (see Theorem \ref{Talboom1}). Moreover, to utilize the properties of derivations on determinants in the actions of $W_{n-1}$ and $S_n$ on tensor modules, we formally extend the derivations
\[
d_i = t_i \frac{\partial}{\partial t_i} \ : \ t_i^n \mapsto n t_i^n,
\]
originally defined for $n \in \mathbb{N}$, to include $n \in \mathbb{C}$.
Section 6 focuses on reducing the Balanced Equality for $W_{n-1}$- and $S_n$-modules into more convenient verification equalities (see Equalities \eqref{eq5.6}-\eqref{eq5.7} and \eqref{eq7.88}-\eqref{eq7.99}).
The final sections, 7 and 8, are devoted to identifying the irreducible cuspidal $W_{n-1}$- and $S_n$-modules from the descriptions that satisfy the simplified equalities, ultimately leading to the characterization of irreducible cuspidal $W^n$- and $S^n$-modules (see Theorems \ref{thm6.7} and \ref{thm8.3}).

\section{Preliminaries}

\

In this section we provide definitions and preliminary results that will be used throughout of the paper.

\begin{defn} \cite{Loday} A vector space $\cal L$ over a field $\mathbb{F}$ equipped with a bilinear bracket operation $[ \cdot , \cdot ]: \cal L \otimes \cal L \to \cal L$ is called Leibniz algebra if for any $x, y, z \in \cal L$ the following (Leibniz) identity
$$[x, [y, z]] = [[x, y], z] - [[x,z],y]$$
holds true.
\end{defn}

Note that a Leibniz algebra satisfying the identity $[x,y]=-[y,x]$ is a Lie algebra (this identity enables the transformation of the Leibniz identity into the Jacobi identity).

\begin{defn} \cite{Filippov0} A vector space $\cal L$ over a field $\mathbb{F}$ equipped with skew-symmetric $n$-ary bracket $[-, \dots, -]$ is called a $n$-Lie algebra if the following fundamental identity holds for any $x_i, y_i\in L$:

\begin{equation}\label{identity1}
[x_1,\dots, x_{n-1},[y_1,\dots, y_n]]=\sum_{i=1}^{n}[y_1, \dots, [x_1,\dots, x_{n-1}, y_i], \dots, y_n].
\end{equation}
\end{defn}

We present some examples of $n$-Lie algebras.

\begin{exam}\label{exam1.2}

(i) Let \( V \) be an \( (n+1) \)-dimensional vector space over a field \( \mathbb{F} \), equipped with a non-zero \( n \)-ary bracket given by
\[
O^n: \quad [e_1, \dots, e_{i-1}, e_{i+1}, \dots, e_{n+1}] = (-1)^{n+1+i} e_i, \quad 1 \leq i \leq n+1,
\]
where \( \{e_1, \dots, e_{n+1}\} \) is a basis of \( V \). Then \( (V, [-, \dots, -]) \) forms an \( n \)-Lie algebra.

(ii) Consider the algebra \( \mathcal{A}_n = \mathbb{C}[t_1^{\pm 1}, \dots, t_n^{\pm 1}] \) of Laurent polynomials, and let the derivations \( d_i \) be defined as
\[
d_i = t_i \frac{\partial}{\partial t_i}, \quad 1 \leq i \leq n.
\]
Define the \( n \)-ary bracket \( \{ -, \dots, - \} \) as
\begin{equation}
\{f_1, \dots, f_n\} = \mathrm{Jac}(f_1, \dots, f_n) =
\begin{vmatrix}
 d_1(f_1) & \dots & d_1(f_n) \\
 \vdots   & \ddots & \vdots \\
 d_n(f_1) & \dots & d_n(f_n)
\end{vmatrix},
\end{equation}
where \( f_1, \dots, f_n \in \mathcal{A}_n \). Then \( S^n = (\mathcal{A}_n, \{ -, \dots, - \}) \) is an \( n \)-Lie algebra.

(iii) Let \( W^n = (\mathcal{A}_n, [-, \dots, -]) \), where the \( n \)-ary bracket is defined as the Wronskian determinant:
\begin{equation}\label{wronskian}
[f_1, \dots, f_n] = \mathrm{Wr}(f_1, \dots, f_n) =
\begin{vmatrix}
 f_1          & \dots & f_n \\
 d_1(f_1)     & \dots & d_1(f_n) \\
 \vdots       & \ddots & \vdots \\
 d_{n-1}(f_1) & \dots & d_{n-1}(f_n)
\end{vmatrix},
\end{equation}
where \( f_1, \dots, f_n \in \mathcal{A}_n \) and \( d_i = t_i \frac{\partial}{\partial t_i} \). Then \( W^n \) is an \( n \)-Lie algebra.
\end{exam}
Note that these \( n \)-Lie algebras are simple \cite{Kac, Dzhum0, Filippov}. Furthermore, in any finite dimension, there is no simple \( n \)-Lie algebra except \( O^n \) (see \cite{Ling}).

It is remarkable that the \( n \)-brackets \( \{ -, \dots, - \} \) and \( [ -, \dots, - ] \) satisfy the properties of \( n \)-Poisson and \( n \)-contact brackets:
\begin{equation}\label{Wronskianproperty}
\begin{array}{lll}
\{fg,f_2, \dots, f_n\}&=&f\{g, f_2, \dots, f_n\}+g\{f, f_2, \dots, f_n\},\\[3mm]
[fg,f_2, \dots, f_n]&=&f[g, f_2, \dots, f_n]+g[f, f_2, \dots, f_n]-fg[1, f_2, \dots, f_n],\\[3mm]
\end{array}
\end{equation}
 respectively (see in \cite{Zelmanov}).

Similarly, just as the concept of Leibniz algebras generalizes Lie algebras \cite{Casas}, a generalization of \( n \)-Lie algebras was introduced. Namely, a \textit{Leibniz \( n \)-algebra} is a vector space over a field \( \mathbb{F} \), equipped with an \( n \)-ary bracket \( [-, \dots, -] \) (not necessarily skew-symmetric) that satisfies the identity \eqref{identity1}. For examples and results on Leibniz \( n \)-algebras, we refer the reader to \cite{Albeverio, Casas0, Casas1, Casas3, omirov, Gago} and the references therein.

\section{Lie and Leibniz structures on $\wedge^{n-1}\cal L$ and $\otimes^{n-1}\cal L$.}

\

In this section, we present a comprehensive survey on Leibniz and Lie structures that can be formulated on the spaces \( \otimes^{n-1} \mathcal{L} \) and \( \wedge^{n-1} \mathcal{L} \). Our goal is to integrate and consolidate results from various sources, including \cite{erratum}, \cite{Balibanu}, \cite{Ladra}, \cite{Takhtajan}, \cite{Boyallian}, and \cite{Boyallian1}, into a unified framework.

\

\subsection{Structures related to Leibniz $n$-algebras}

\

\

Let \( \mathcal{L}= (\mathcal{L}, [ -, \dots, - ]) \) be a Leibniz \( n \)-algebra. For arbitrary elements \( x = x_1 \otimes \dots \otimes x_{n-1} \) and \( y = y_1 \otimes \dots \otimes y_{n-1} \) in the space \( \otimes^{n-1} \mathcal{L} \), we define the binary bracket
\[
[x,y]_{\otimes} = \sum_{i=1}^{n-1} y_1 \otimes \dots \otimes [x_1, \dots, x_{n-1}, y_i] \otimes \dots \otimes y_{n-1}.
\]
Then, \( \otimes^{n-1} \mathcal{L}\) with respect to the bracket \( [ -, - ]_{\otimes} \) forms a Leibniz algebra \cite{Casas}. We call \( BLLb^{\otimes}_{n-1} := (\otimes^{n-1} \mathcal{L}, [ -, - ]_{\otimes}) \) a \textit{basic Leibniz algebra associated with a Leibniz \( n \)-algebra \( \mathcal{L} \)}.

For a given \( n \)-Lie algebra \( \cal L \) and \( x \in \otimes^{n-1} \mathcal{L} \), we define an operator \( \text{ad}_{\otimes} : \otimes^{n-1} \cal L \rightarrow \cal L \) by
\[
{\text{ad}_x}_\otimes (y) = [x_1, \dots, x_{n-1}, y], \quad y \in \cal L.
\]
From \eqref{identity1}, we derive that the kernel of the adjoint action
\[
\text{Ker}(\text{ad}_{\otimes}) = \text{Span} \{ x \in \otimes^{n-1} \mathcal{L} \mid {\text{ad}_x}_\otimes = 0 \}
\]
is an ideal of the Leibniz algebra \( BLLb^{\otimes}_{n-1} \).

Similarly to \cite{Ladra}, one can prove the following result.

\begin{prop} \label{prop2.2}
Let \( (\mathcal{L}, [ -, \dots, - ]) \) be a Leibniz \( n \)-algebra. Then the equality
\[[{\text{ad}_x}_\otimes, {\text{ad}_y}_\otimes] = \text{ad}_{[x,y]_{\otimes}} \quad \mbox{for any} \quad x, y \in \otimes^{n-1} \cal L \]
 holds.
\end{prop}

One can check that \( \mathcal{I}_{\otimes} = \text{Span} \{ [x, y]_{\otimes} + [y, x]_{\otimes} \mid x, y \in \otimes^{n-1} \mathcal{L} \} \)
is an ideal of the Leibniz algebra \( BLLie^{\otimes}_{n-1} \). A Lie algebra \( BLLb^{\otimes}_{n-1} / \mathcal{I}_{\otimes} \) is said to be \textit{a Lie algebra associated with a basic Leibniz algebra \( BLLb^{\otimes}_{n-1} \)}.

\begin{lem} \label{idealI}
We have the embedding: \( \mathcal{I}_{\otimes} \subseteq \text{Ker}(\text{ad}_{\otimes}) \).
\end{lem}

\begin{proof}
Let \( x, y \in \otimes^{n-1} \mathcal{L} \) and \( z \in \mathcal{L} \), and let \( v = [x, y]_{\otimes} + [y, x]_{\otimes} \). Then, by straightforward application of \eqref{identity1}, we obtain \( {\text{ad}_v}_\otimes (z) = 0 \).
\end{proof}

By virtue of Lemma \ref{idealI}, we derive that the quotient algebra \( BLLb^{\otimes}_{n-1} / \text{Ker}(\text{ad}_{\otimes}) \) is a subalgebra of the Lie algebra \( BLLb^{\otimes}_{n-1} / \mathcal{I}_{\otimes} \).

Consider the subspace
\[
{\cal W_{n-1}}_{\otimes} = \text{Span} \{ x_1 \otimes \cdots \otimes x_{n-1} \mid x_i = x_j \ \text{for some} \ 1 \leq i < j \leq n-1 \}
\]
of the space \( \otimes^{n-1} \mathcal{L} \).

Unlike the \( n \)-Lie algebra case, we have \( {\cal W_{n-1}}_{\otimes} \nsubseteq \text{Ker}(\text{ad})_{\otimes} \), and in general, \( {\cal W_{n-1}}_{\otimes} \) is not an ideal of the basic Leibniz algebra \( BLLb^{\otimes}_{n-1} \). This can be verified, for example, in the 2-dimensional Leibniz 3-algebra with the multiplication table:
\[
[e_2, e_2, e_1] = -e_1, \quad [e_2, e_1, e_2] = e_1, \quad [e_2, e_2, e_2] = e_1.
\]

\begin{rem} \label{rem2.3}
Note that, in general, we have \( {\cal W_{n-1}}_{\otimes} \nsubseteq \mathcal{I}_{\otimes} \) and \( \mathcal{I}_{\otimes} \nsubseteq {\cal W_{n-1}}_{\otimes} \). Indeed, consider a Leibniz \( n \)-algebra \( \mathcal{L} \) such that \( \mathcal{L} \neq \mathcal{L}^2 \) (for instance, if \( \mathcal{L} \) is nilpotent, solvable, or admits the form \( \mathcal{L}' \oplus \mathbb{F}^k \)). Let \( e \in \mathcal{L} \setminus \mathcal{L}^2 \). Then it is easy to see that \( e \otimes \cdots \otimes e \in {\cal W_{n-1}}_{\otimes} \setminus \mathcal{I}_{\otimes} \), and \( [x, y]_{\otimes} + [y, x]_{\otimes} \neq e \otimes \cdots \otimes e \) for any \( x, y \in \otimes^{n-1} \mathcal{L} \).
\end{rem}

In the following diagram, we summarize our observations regarding structures related to Leibniz \( n \)-algebras

\begin{center}
\begin{tikzcd}
&

  \mbox{Leibniz n-alg.} \ \cal L
  \arrow{r}
  & \begin{array}{ccc}
    \mbox{Leibniz alg.}\\[1mm]
  \shortparallel\\[1mm]
  BLLb^{\otimes}_{n-1}
  \end{array}
  \arrow{r}
  &
  \begin{array}{ccc}
    \mbox{Lie alg.}\\[1mm]
  \shortparallel\\[1mm]
  BLLb^{\otimes}_{n-1}/Ker(ad_{\otimes})
  \end{array}
  \arrow[hookrightarrow]{r}
  &
\begin{array}{ccc}
    \mbox{Lie alg.}\\[1mm]
  \shortparallel\\[1mm]
  BLLb^{\otimes}_{n-1}/\cal I_{\otimes}
  \end{array}
\end{tikzcd}
\end{center}

\

\subsection{Structures related to $n$-Lie algebras}

\

\

Consider an $n$-Lie algebra $\cal L=(\cal{L}, [ - , \dots, - ])$. Then a vector space $\otimes^{n-1}\cal L$ with respect to the bracket
$[x,y]_{\otimes}$ forms a Leibniz algebra \cite{Takhtajan}.  We call the algebra $BLLie^{\otimes}_{n-1}:=(\otimes^{n-1}\cal{L}, [ - , - ]_{\otimes})$ {\it a basic Leibniz algebra associated with a $n$-Lie algebra $\cal L.$}

In \cite{Takhtajan} it was proved that $Ker(ad)_{\otimes}$ is an ideal of $BLLie^{\otimes}_{n-1}$ and $BLLie^{\otimes}_{n-1}/ Ker(ad)_{\otimes}$ is a Lie algebra, which is called {\it a basic Lie algebra associated with a $n$-Lie algebra $\cal L_{Lie}$}.

One can consider another quotient Lie algebra $BLLie^{\otimes}_{n-1}/ \cal I_{\otimes}$ as well. This algebra is said to be {\it a Lie algebra associated with a basic Leibniz algebra $BLLie^{\otimes}_{n-1}$.}

Similar to Lemma \ref{idealI} one can check that $\cal I_{\otimes}\subseteq  Ker(ad_{\otimes})$ and hence, $BLLie^{\otimes}_{n-1}/ Ker(ad_{\otimes})$ is a subalgebra of a Lie algebra $BLLie^{\otimes}_{n-1}/ \cal I_{\otimes}$.

\begin{prop} \cite{Ladra} Let $\cal L_{Lie}$ be a $n$-Lie algebra. Then for any $x,y\in \otimes^{n-1}\cal L$ the equality

$$[{ad_x}_{\otimes},{ad_y}_{\otimes}]=ad_{[x,y]_{\otimes}}=ad_{x\circ y} \quad \mbox{with} \quad x\circ y=\frac{1}{2}([x,y]_{\otimes}-[y,x]_{\otimes})$$
holds.
\end{prop}

\begin{defn}\label{skew}
We say that the product $[- ,  - ]_{\otimes}$ is skew-symmetric if $[x,y]_{\otimes}=-[y,x]_{\otimes}$ for any $x,y\in \otimes^{n-1}\cal L$.
\end{defn}

Note that $[- ,  - ]_{\otimes}$ is skew-symmetric in the case of $n$-Lie algebra $O^n$ (see \cite{erratum}), which implies that for $O^n$ the algebra $BLLie^{\otimes}_{n-1}$ is Lie algebra. However, this is not true for $n$-Lie algebras $S^n$ and $W^n$.

\begin{rem} \label{exam04}
Simple computations show that, in the case of $W^3$, for the elements \( x = t_1 \otimes t_2 \) and \( y = t_1^2 \otimes t_2^2 \) in \( BLLie^{\otimes}_{n-1} \), we obtain \( [x,y]_{\otimes} \neq -[y,x]_{\otimes} \). Similarly, in the case of \( S^3 \), for the elements \( x = t_1 \otimes t_2 \) and \( y = t_1^2 t_3 \otimes t_2^2 t_3 \) in \( BLLie^{\otimes}_{n-1} \), it also holds that \( [x,y]_{\otimes} \neq -[y,x]_{\otimes} \).

Therefore, Proposition 1 in \cite{Boyallian} and \cite{Boyallian1} is incorrect, even though those results were established in the linearly compact case. This discrepancy arises because the basic Leibniz algebras \( BLLie^{\otimes}_{n-1} \), associated with the \( n \)-Lie algebras \( W^n \) and \( S^n \), are non-Lie Leibniz algebras under the product \( [x,y]_{\otimes} \).
\end{rem}

Consider the subspace
$${\cal W_{n-1}}_{\otimes}=Span\{x_1\otimes\cdots \otimes x_{n-1} \ | \ x_i=x_j \ \mbox{for some} \ 1\leq i<j\leq n-1\}$$
of the space $\otimes^{n-1}\cal L$.

In \cite{Ladra} it was proved that ${\cal W_{n-1}}_{\otimes}$ is an ideal of a basic Leibniz algebra $BLLie^{\otimes}_{n-1}$. Similar to Remark \ref{rem2.3} one can get
${\cal W_{n-1}}_{\otimes}\nsubseteq \cal I$ and $\cal I\nsubseteq {\cal W_{n-1}}_{\otimes}$.
Further, we shall use notations
$$[x,y]_{\wedge}, \quad ad_{\wedge}, \quad \cal I_{\wedge}$$
that are obtained from the corresponding notations by replacing $\otimes$ to $\wedge$.

We denote by $BLLie^{\wedge}_{n-1}$ the Leibniz $n$-algebra $BLLie^{\otimes}_{n-1}/ {\cal W_{n-1}}_{\otimes}=(\wedge^{n-1}L, [ - , - ]_{\wedge}).$

It is obvious that ${\cal W_{n-1}}_{\otimes}\subseteq Ker(ad_{\otimes})$. Therefore, $BLLie^{\otimes}_{n-1}/ Ker(ad_{\otimes})$ is a Lie subalgebra of a Leibniz algebra $BLLie^{\wedge}_{n-1}.$

Let us introduce some notations

\begin{defn} A Lie algebra $\widehat{\cal L}:=BLLie^{\wedge}_{n-1}/ \mathcal{I}_{\wedge}$ with the product $[x,y]_{\wedge}$
is called {{\it the universal Lie algebra} associated with $n$-Lie algebra $\cal L$} and a Lie algebra $\widetilde{\cal L}:=BLLie^{\wedge}_{n-1}/ Ker(ad_{\wedge})$ is called  the Lie algebra associated with $n$-Lie algebra $L$.
\end{defn}

Thanks to Lemma \ref{idealI} we get $\widetilde{\cal L}$ is subalgebra of the Lie algebra $\widehat{\cal L}.$ Taking into account
$$[{ad_x}_{\wedge},{ad_y}_{\wedge}]={ad_x}_{\wedge}\circ {ad_y}_{\wedge} - {ad_y}_{\wedge}\circ {ad_x}_{\wedge}=
{ad_{[x,y]_{\wedge}}}_{\wedge},$$
we conclude that $ad$ is an epimorphism from Leibniz algebra $BLLie^{\wedge}_{n-1}$ onto the Lie algebra $InDer(\cal L)$. Therefore,
\begin{equation}\label{Inder}
InDer(\cal L)\cong BLLie^{\wedge}_{n-1}/Ker(ad_{\wedge}).
\end{equation}

It is remarkable that $Ker(ad_{\wedge}) \subseteq Center(\widehat{L})$ and $Ker(ad_{\wedge}) \subseteq Ann_l(BLLie^{\wedge}_{n-1}).$

\begin{rem} \label{counterexample} It should be noted that results in \cite{Boyallian} and \cite{Boyallian1} on descriptions of $Ker(ad_{\wedge})$ in the cases of $n$-Lie algebras $W^n$ and $S^n$ are not correct (also for linearly compact case). Indeed, in our case of $W^3$ we have
$f=5t_1^{4}t_2^{4}\wedge t_1^{6}t_2^{5}+t_1^{10}t_2^{7}\wedge t_2^2\in Ker(ad_{\wedge}).$ Similarly, taking $f=6t_1^{2}t_2^{4}t_3\wedge t_1^{3}t_2^{5}t_3-t_1^{2}t_2^{6}t_3^{2}\wedge t_1^{3}t_2^{3}\in S^3 \wedge S^3,$ one can check that $f\in Ker(ad_{\wedge})$.
\end{rem}

%
%
%

The following Lemma was proved in \cite{Boyallian1}.

\begin{lem} \label{lem3.4} If $a\in Ker(ad)$ and $\rho$ is a representation of the Lie algebra $\wedge^{n-1}L$, then $\rho(a)$  commutes with $\rho(b)$ for any $b\in \wedge^{n-1}L.$
\end{lem}

Since an algebra $\wedge^{n-1}L$ is non-Lie Leibniz algebras, in general, Lemma \ref{lem3.4} does not hold. However, taking into account that representation of Leibniz algebra is a pair $(\rho, \lambda)$ that satisfies the conditions:
$$\begin{array}{lllll}
1. &\rho([x,y]) = \rho(y)\rho(x) - \rho(x)\rho(y),\\[1mm]
2. &\lambda([x,y]) = \rho(y)\lambda(x) - \lambda(x)\rho(y),\\[1mm]
3. &\lambda([x,y]) = \rho(y)\lambda(x) + \lambda(x)\lambda(y).\\[1mm]
\end{array}
$$

In our notation $\wedge^{n-1}L$ means $BLLie^{\wedge}_{n-1}$. More details on representations of a Leibniz algebras can be find in \cite{Loday}. Thus, Lemma \ref{lem3.4} can be formulated in the following correct form.

\begin{lem} \label{lem3.5} If $x\in Ker(ad_{\wedge})$ and $(\rho, \lambda)$ is a representation of the Leibniz algebra $BLLie^{\wedge}_{n-1}$, then $\rho(x)$ and $\lambda(x)$ commute with $\rho(y)$ for any $y\in BLLie^{\wedge}_{n-1}.$ Moreover, $\lambda(x)\rho(y)=- \lambda(x)\lambda(y).$
\end{lem}
\begin{proof} Let $x\in Ker(ad)$. Then
$$\begin{array}{llllll}
0=\rho(ad_x(y))=\rho([x,y])=\rho(x)\rho(y)-\rho(y)\rho(x),\\[1mm]
0=\lambda(ad_x(y))= \lambda([x,y]) = \rho(y)\lambda(x) - \lambda(x)\rho(y),\\[1mm]
0=\lambda(ad_x(y))=\lambda([x,y]) = \rho(y)\lambda(x) + \lambda(x)\lambda(y).\\[1mm]
\end{array}$$
\end{proof}

In the following diagram we summarize our observations regarding structures related to $n$-Lie algebras:
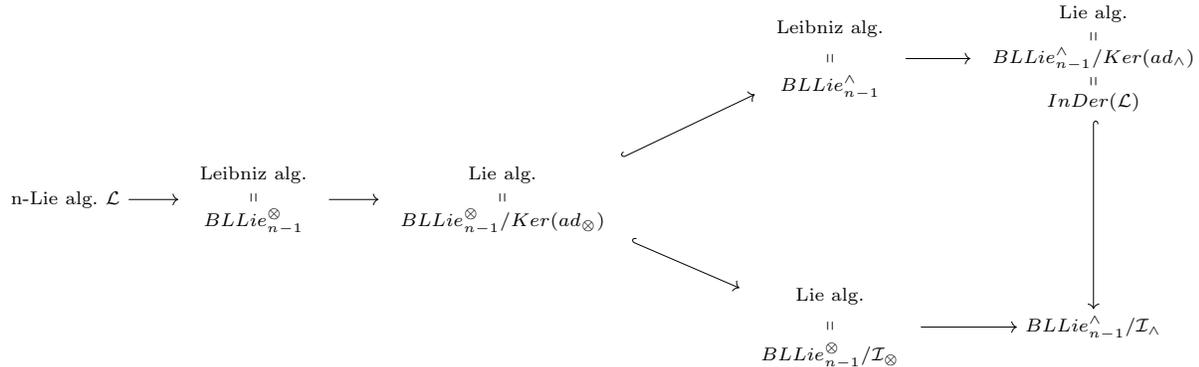
\begin{center}
\scriptsize{
\begin{tikzcd}
  &
  &
  &   \begin{array}{ccc}
    \mbox{Leibniz alg.}\\[1mm]
  \shortparallel\\[1mm]
  BLLie^{\wedge}_{n-1}
  \end{array}\arrow{r}
  &
      \begin{array}{ccc}
    \mbox{Lie alg.}\\[0.1mm]
  \shortparallel\\[0.1mm]
BLLie^{\wedge}_{n-1}/Ker(ad_{\wedge})\\[0.1mm]
\shortparallel\\[0.1mm]
InDer(\cal L)
  \end{array}

  \arrow[hookrightarrow]{dd}
  &

 \\
  \mbox{n-Lie alg.} \ \cal L
  \arrow{r}
  &
  \begin{array}{ccc}
    \mbox{Leibniz alg.}\\[0.1mm]
  \shortparallel\\[0.1mm]
BLLie^{\otimes}_{n-1}
  \end{array}\arrow{r}
  &
    \begin{array}{ccc}
    \mbox{Lie alg.}\\[0.1mm]
  \shortparallel\\[0.1mm]
BLLie^{\otimes}_{n-1}/Ker(ad_{\otimes})
  \end{array} \quad \quad \quad \arrow[hookrightarrow]{ur}\arrow[hookrightarrow]{dr}
  &
  & \\
  &
  &
  &
\begin{array}{ccc}
    \mbox{Lie alg.}\\[1mm]
  \shortparallel\\[1mm]
  BLLie^{\otimes}_{n-1}/\cal I_{\otimes}
  \end{array}\arrow{r}
  &BLLie^{\wedge}_{n-1}/\cal I_{\wedge}
\end{tikzcd}}
\end{center}

\


\section{Modules over $n$-Lie algebras}

\

In this section we give the notion of module over an $n$-Lie algebra $\cal L$ and its relation with module over Lie algebra $InDer(\cal L)$.

%

\begin{defn} A vector space $M$ is called a module over an $n$-Lie algebra $\cal L$ is $\cal L\oplus M$ is $n$-Lie algebra such that $\cal L$ is its subalgebra and $M$ is an abelian ideal.
\end{defn}
In other words, the vector space $\cal L\oplus \cal M$ can be equipped with an $n$-ary product $[-, \dots, -]$ such that
$[x_1, \dots, x_n]=0$ if at least two arguments belong to $\cal M$ and

\begin{equation}\label{mod1}
[x_1, \dots, x_i, m, x_{i+1}, \dots, x_n]=-[x_1, \dots, x_i, x_{i+1}, m, \dots, x_n], \ 1\leq i <n,
\end{equation}

\begin{equation}\label{mod2}\begin{array}{cc}
[x_1, \dots, x_{n-1}, [y_1, \dots, y_{n-1}, m]]- [y_1, \dots, y_{n-1}, [x_1, \dots, x_{n-1}, m]]=&\\
\sum\limits_{i=1}^{n-1}[y_1, \dots, y_{i-1}, [x_1, \dots, x_{n-1},y_i] ,y_{i+1}, \dots, y_{n-1},m],&
\end{array}
\end{equation}

\begin{equation}\label{mod3}
[x_1, \dots, x_{n-2},[y_1, \dots, y_{n}],m]=
\sum_{i=1}^{n}[y_1, \dots, y_{i-1}, [x_1, \dots, x_{n-2},y_i,m], y_{i+1}, \dots, y_n].
\end{equation}

Note that Equality (\ref{mod2}) can be rewritten as follows:

\begin{equation}\label{mod4}
ad_{x_1 \wedge \dots \wedge x_{n-1}} ad_{y_1\wedge \dots \wedge y_{n-1}}(m)-
ad_{y_1\wedge \dots \wedge y_{n-1}}ad_{x_1 \wedge \dots \wedge x_{n-1}}(m)=
ad_{ad_x(y)}(m).
\end{equation}

Now we shall try to find the correspondence between modules of the Lie algebra $\widetilde{\cal L}\simeq InDer(\cal L)$ and $n$-Lie algebra $\cal L$.

We recall, that $\cal N$ is a module over $\widetilde{\cal L}$ if $\widetilde{\cal L}\oplus \cal N$ is a Lie algebra, that is, $[x,y]*n=ad_x(y)*n=x*(y*n)-y*(x*n)$ and $x*n=-n*x$.

From Equality (\ref{mod4}) we can conclude that any module $\cal M$ over $n$-Lie algebra $\cal L$ is a module over Lie algebra $\widetilde{\cal L}$ with the action:
$$x*m=ad_x*m=[x_1, \dots, x_{n-1},m].$$

Indeed, from Equality (\ref{mod4}) we have
$$[x,y]*m=ad_x(y)*m=
ad_x*(ad_y*m)-ad_y* (ad_x*m)=x *(y*m)-y*(x*m).$$

Let us rewrite Equality (\ref{mod3}) in the following form:

$$[ x_1, \dots, x_{n-2},[y_1, \dots, y_{n}],m]=
\sum\limits_{i=1}^{n}(-1)^{n+i}[y_1, \dots, y_{i-1}, y_{i+1}, \dots, y_n, [x_1, \dots, x_{n-2},y_i,m]].$$

In terms of the action on module it has the form (the Balanced Equality):
\begin{equation}\label{eq3.7}
ad_{x_1 \wedge \dots \wedge x_{n-2}\wedge [y_1, \dots, y_{n}]} *m=
\sum\limits_{i=1}^{n}(-1)^{n+i} ad_{y_1 \wedge \dots \wedge y_{i-1} \wedge y_{i+1} \wedge \dots \wedge y_n} * (ad_{x_1 \wedge \dots \wedge x_{n-2}\wedge y_i} *m).
\end{equation}

We denote by $\varphi$ an isomorphism between the Lie algebras $\widetilde{\cal L}$ and $InDer(\cal L).$

Let $M$ be a module over the Lie algebra $InDer(\cal L)$. Then we define on $M$ the $n$-Lie module structure over $\cal L$ as follows:

\begin{equation}\label{eqaction}
\begin{array}{llll}
[x_1, \dots, x_{i-1}, m, x_{i+1},\dots, x_n]:= -
  [x_1, \dots, x_{i}, m, x_{i+2}, \dots, x_n], \quad i<n;\\[3mm]
[x_1, \dots, x_{n-1}, m]:=ad_{x_1\wedge \dots \wedge x_{n-1}}(m), \quad where \quad ad_{x_1\wedge \dots \wedge x_{n-1}}(m)=\varphi(x_1\wedge \dots \wedge x_{n-1})(m).\\[3mm]
\end{array}
\end{equation}

Summarizing the arguments used above we obtain the relation between modules over $n$-Lie algebra $\cal L$ and modules over its Lie algebra $\widetilde{\cal L}$.
\begin{prop} \label{prop3.3} Any module $\cal M$ over $n$-Lie algebra $\cal L$ is a module over Lie algebra $\widetilde{\cal L}$ with the action
$x*m=[x_1, \dots, x_{n-1},m]$ which satisfies Equality (\ref{eq3.7}) and any module over Lie algebra $\widetilde{\cal L}$ which satisfy Equality (\ref{eq3.7}) is a module over $n$-Lie algebra $\cal L$.
\end{prop}

%
%


Similar to \cite{Dzhum} one can establish that an $n$-Lie module $\cal M$ of $\cal L$ is irreducible/completely reducible if and only if it is irreducible/completely reducible as a module of the Lie algebra
$\widetilde{\cal L}$.


\

\section{Irreducible cuspidal $W_{n-1}$ and $S_n$-modules.}\label{Sec5}

\

In this section we give definitions and some preliminary results on cuspidal modules over Lie algebras $W_{n-1}$ and $S_n$, as well their relations with $InDer(W^{n})$ and $InDer(S^{n})$, respectively.

Denote by $\cal A_{n}=\mathbb{C}[t_1^{\pm 1}, \dots, t_{n}^{\pm 1}]$ the algebra of Laurent polynomials and consider the column vector space $\mathbb{C}^{n}$ with standard basis
$\{e_1, \dots, e_{n}\}$. Let $(\cdot | \cdot)$  be symmetric bilinear form $(u|v)=u^{T}v\in \mathbb{C}$, where $u, v \in \mathbb{C}^{n}$ and $u^{T}$ denotes matrix transpose.

Homogeneous (in the power of $t$) elements of $Der(\cal A_{n})$ we shall denote by $D(u, r)=\sum\limits_{i=1}^{n}u_it^{r}d_i,$ where $u\in \mathbb{C}^{n}, t^r=t_1^{r_1}\dots t_{n}^{r_{n}}$ for $r=(r_1, \dots, r_{n})\in \mathbb{Z}^{n}$ and $d_i=t_i\frac{\partial}{\partial t_i}$. Then the Lie bracket is given by
\begin{equation}\label{eqLiebracket}
[D(u, r), D(v, s)]=D((u|s)v - (v|r)u, r +s)
\end{equation}
for $u, v \in \mathbb{C}^{n}$ and $r, s\in \mathbb{Z}^{n}$.

The Lie algebra $Der(\cal A_{n})$ (denoted by $W_{n}$) has a natural structure of an $\cal A_{n}$-module of free rank $n$.

Note that geometrically, $W_{n}$ may be interpreted as the Lie algebra of (complex-valued) polynomial vector fields on an $n$-dimensional torus via the mapping $t_i= e^{\sqrt{-1}\theta_i}$ all $i\in \{1, \dots, n\}$, where $\theta_i$ is the $i$th angular coordinate. In fact, the change of coordinates $t_i=e^{\sqrt{-1}\theta_i}$, gives $\sqrt{-1}d_i=\sqrt{-1}t_i\frac{\partial}{\partial t_i}=\frac{\partial t_i}{\partial \theta_i}\cdot \frac{\partial}{\partial t_i}=\frac{\partial}{\partial \theta_i}.$ Therefore, an element $X =\sum\limits_{i=1}^{n}f_id_i\in W_{n}$ can be written in the form $X=\sqrt{-1}\sum\limits_{i=1}^{n}f_i\frac{\partial}{\partial \theta_i}$.

It is well-known that $W_{n}=\bigoplus\limits_{i=1}^{n}\cal A_{n}d_i,$ with $d_i=t_i\frac{\partial}{\partial t_i}, \ 1\leq i \leq n.$
Then \eqref{eqLiebracket} has the form
$$[t^rd_i, t^s d_j]=s_it^{r+s}d_j - r_jt^{r+s}d_i, \ r, s\in \mathbb{Z}^{n}, \ 1\leq i, j \leq n.$$
The subspace spanned by $\{d_1, \dots, d_n\}$ is the Cartan subalgebra in $W_n$  and the adjoint action of these elements induces a $\mathbb{Z}^n$-grading.

It is remarkable that the Lie algebra $W_{n}$ has an interesting subalgebra, the Lie algebra of divergence-zero vector fields, denoted $S_{n}$.
The divergence of $X$ w.r.t. the natural volume form in angular coordinates is then
$-\sqrt{-1}\sum\limits_{i=1}^{n}\frac{\partial f_i}{\partial \theta_i}=
\sum\limits_{i=1}^{n}t_i\frac{\partial f_i}{\partial t_i}.$

In \cite{Talboom} it is proved that $D(u,r) \in S_{n}$ if and only if $(u|r)=0.$ Moreover, for $D(u,r), D(v,s)\in S_{n}$,
$((u|s)v - (v|r)u | r + s) = 0,$ demonstrating that $S_{n}$ is closed under the Lie bracket.

Setting $d_{ij}(r)=r_jt^rd_i - r_it^rd_j,$
we get $S_{n}=Span\{d_i, d_{ij}(r) \ | \ 1\leq i, j \leq n, r\in \mathbb{Z}^{n}\}.$  The Lie bracket of $S_{n}$ in terms of the elements $d_{ij}(r)$ has the following form
$$\begin{array}{llllll}
[d_i, d_{pq}(r)]&=&r_id_{pq}(r),\\[3mm]
[d_{ij}(r), d_{pq}(s)]&=&r_js_pd_{iq}(r + s) - r_js_qd_{ip}(r+s) -
r_is_pd_{jq}(r + s) + r_is_qd_{jp}(r + s),
\end{array}$$
with $r,s\in \mathbb{Z}^{n}$ and $1\leq i, j, p, q \leq n$.

In addition, we assume $d_{ij}(0)=0, \ d_{ii}(r)=0$ and $d_{ji}(r)= -d_{ij}(r)$. For $n\geq 3$ $d_{ij}(r)=0$ in the case $r_i=r_j=0$ and in general,
$r_pd_{ij}(r)+r_{i}d_{jp}(r)+r_jd_{pi}(r)=0.$

Note that
$$\begin{array}{llllll}
d_{i}&=&D(u,0) & with & u^{T}=(0, \dots, 0, \underbrace{1}_{i}, 0, \dots, 0), \\[3mm]
d_{ij}(r)&=&D(u,r) & with & u^{T}=(0, \dots, 0, \underbrace{r_j}_{i}, 0, \dots, 0, \underbrace{-r_i}_{j}, 0, \dots, 0)\\[3mm]
\end{array}$$

Note that the Cartan subalgebras of $S_n$ and $W_n$ are coincide. A $W_n$-module $M$ (respectively, a $S_n$-module) is called {\it a weight module} if $M=\bigoplus\limits_{\lambda\in \mathbb{C}^n}M_{\lambda}$, where the weight space $M_{\lambda}$ is defined as $M_{\lambda}=\{v \in M \ | \  d_iv = \lambda_i v \quad \mbox{for} \quad i=1, \dots, n\}$.
In particular, $W_n$ and $S_n$ are weight modules over themselves. A weight module is called {\it cuspidal} if the dimensions of its weight spaces are uniformly bounded by some constant.

\

\subsection{Irreducible cuspidal $W_{n-1}$-modules}

\

\

Let $U$ be a finite-dimensional simple $gl_{n-1}$-module and let $\alpha\in \mathbb{C}^{n-1}$. Then the module of
tensor fields $T(U, \alpha)$ is the space
$$T(U, \alpha)=t^{\alpha}\cal A_{n-1}\otimes U$$
with a $W_{n-1}$-action given as follows:
\begin{equation}\label{action0}
t^{m}d_j*(t^s\otimes u) =s_jt^{s+m}\otimes u + \sum_{i=1}^{n-1}m_it^{s+m}\otimes E_{i,j}u,
\end{equation}
where $m\in \mathbb{Z}^{n-1}, s \in \alpha + \mathbb{Z}^{n-1}, u \in U$ and $E_{i,j}$ are matrix units.

Although $d_j=t^j\frac{\partial}{\partial t_j},$ for more convenient computational purposes we shall formally assume that $d_j(t^s)=s_jt^s$ for $s \in \alpha + \mathbb{Z}^{n-1}$. This enables us to reformulate the action \eqref{action0} in the form:
\begin{equation}\label{action1}
t^{m}d_j*(t^s\otimes u) =t^{m}d_j(t^{s})\otimes u + t^s\sum_{i=1}^{n-1}d_i(t^{m})\otimes E_{i,j}u.
\end{equation}

Let $V$ be the natural $(n-1)$-dimensional representation of $gl_{n-1}$ and let $\wedge^sV$ be its $s$-th exterior power ($0\leq s \leq n-1$). Note that $\wedge^sV$ is a simple $gl_{n-1}$-module for all $s=0, \dots, n-1.$ This module has highest weight vector $e_1\wedge \cdots \wedge e_s$ with the standard basis $\{e_1, \dots, e_{n-1}\}$ of $V$ and its action is defined as follows:
$$X*(v_1\wedge \cdots \wedge v_s)=\sum\limits_{i=1}^{k}v_1 \wedge \cdots \wedge Xv_i \wedge \cdots \wedge v_s, \quad X\in gl_{n-1}, \quad v_i\in V$$
along with $E_{i,j}e_l=\delta_{j,l}e_i$.

The modules
$$\Omega^s(\alpha) = T (\wedge^sV, \alpha)=t^{\alpha}\cal A_{\mathbb{C}}\otimes \wedge^sV$$
are the modules of differential $s$-forms. These modules form the de Rham complex
$$\Omega^0(\alpha) \xrightarrow[]{d} \Omega^1(\alpha) \xrightarrow[]{d} \dots \xrightarrow[]{d} \Omega^{n-1}(\alpha).$$
The differential $d$ of the de Rham complex is a homomorphism of $W_{n-1}$-modules.

Recall the action of $d$ on $\Omega^s(\alpha)$:
$$d(t^r\otimes v_1 \wedge \cdots \wedge v_s)=
\sum\limits_{i=1}^{n-1}d_i(t^r)\otimes e_i\wedge v_1 \wedge \cdots \wedge v_s.$$

The following result presents the description of simple $W_{n-1}$-modules with finite-dimensional weight spaces \cite{Slava}.

\begin{thm} \label{mainthm}
\

(1) Every simple $W_{n-1}$-module with finite-dimensional weight spaces
is either cuspidal or of the highest weight type.

(2) A simple cuspidal $W_{n-1}$-module is isomorphic to one of the following:
\begin{itemize}
\item a module of tensor fields $T(U, \alpha)$, where $\alpha\in \mathbb{C}^{n-1}$ and $U$ is a finite-dimensional
simple $gl_{n-1}$-module different from the $s$-th exterior power of the natural $(n-1)$-dimensional
$gl_{n-1}$-module, $0 \leq s \leq n-2;$
\item a submodule $d(\Omega^s(\alpha))\subseteq \Omega^{s+1}(\alpha),  \ 0 \leq s\leq n-2;$
\item a trivial 1-dimensional $W_{n-1}$-module.
\end{itemize}

(3) A module of the highest weight type is isomorphic to $L(U, \alpha, \gamma)^g$ for some finite-dimensional simple $gl_{n-2}$-module $U,$ $\alpha\in \mathbb{C}^{n-2}, \gamma \in \mathbb{C}$ and $g\in GL_{n-1}(\mathbb{Z}).$
\end{thm}
Note that $\widetilde{\cal W}^n\cong InDer(\cal W^n)$ and $InDer(\cal W^n)\cong W_{n-1}$ via isomorphism
$$\psi \ : \ ad_{f_1\wedge \dots \wedge f_{n-1}} \to \sum_{k=1}^{n-1}(-1)^{n+1-k}Wr^k_{f_1, \dots, f_{n-1}}d_k,
$$
where $Wr^k_{f_1, \dots, f_{n-1}}$ is the determinant obtained from $[f_1, \dots, f_{n-1}]$  (see  Example \ref{exam1.2}) by deleting its $k$-th row and the last column.

\

\subsection{Irreducible cuspidal $S_{n}$-modules}

\

\

Let $V(\lambda)$ be a finite-dimensional irreducible $sl_{n}$-module of highest weight $\lambda=\sum\limits_{i=1}^{n}a_i\omega_i$ with $a_i\in Z_{\geq 0}$ and $\omega_i$ fundamental dominant weights. Let us consider the restriction of Larsson's functor to $sl_n$-modules. Then
$$F^{\alpha} \ : \ sl_n-\mbox{modules}  \ \rightarrow  \ S_n-\mbox{modules}$$
and $F^{\alpha}(\lambda)=\cal A_n \otimes V(\lambda)$ for $\alpha\in \mathbb{C}^{n}$. In fact, the Larsson's functor $F^{\alpha}$ was defined in \cite{Larsson}, \cite{Sheng} and studied in \cite{Rao}. Homogeneous elements of $F^{\alpha}(\lambda)$ will be denoted $v(s)=t^s\otimes v$ for $v\in V(\lambda), s\in \mathbb{Z}^{n}$. Then $F^{\alpha}(\lambda)$ becomes a $S_{n}$-module via the action
\begin{equation}\label{eq7.1}
D(u,r)*v(s)=(u|s+\alpha)v(s+r) + (ru^{T})v(s+r),
\end{equation}
for $D(u,r)\in S_{n}$. Note that $ru^{T}$ is an $n\times n$ traceless matrix, as $(u|r)=0$ by assumption, and so the second term involves the $sl_{n}$-module action.

Setting $F^{\alpha}(\lambda)=\bigoplus\limits_{s\in \mathbb{Z}^{n}}t^s\otimes V(\lambda)$ and
taking into account that ${d_i}*(t^s\otimes v)=s_it^{s}\otimes v$ we derive that the subspace
$t^s \otimes V(\lambda)$ is homogeneous in $t$ and, hence $F^{\alpha}(\lambda)$ is a graded module w.r.t. the action of the Cartan subalgebra.

Let us rewrite \eqref{eq7.1} in the form presented in \cite{Rao}:

\begin{equation}\label{eq7.2}
D(u,r)*(t^s\otimes v)=(u|s+\alpha)t^{s+r}\otimes v
+ \sum\limits_{i,j=1}^{n}r_iu_jt^{s+r}\otimes E_{i,j}v,
\end{equation}

This action in basis elements of $S_n$ has the following form:
\begin{equation}\label{eq7.33}\begin{array}{lllllllll}
{d_i}*(t^s\otimes v)&=&(s_i+\alpha_i)t^{s}\otimes v, \\[3mm]
d_{ij}(r)*(t^s\otimes v)&=&
\Big(r_j(s_i+\alpha_i)-r_i(s_j+\alpha_j)\Big)t^{s+r}\otimes v\\[3mm]
&+&t^{s+r}\otimes \Big(\sum\limits_{p=1, p\neq i}^{n}r_{p}r_{j}E_{p,i}-
\sum\limits_{p=1, p\neq j}^{n}r_{p}r_{i}E_{p,j}+
r_ir_j(E_{i,i}-E_{j,j})\Big)v.\\[3mm]
\end{array}
\end{equation}

Now we consider the next irreducible $S_n$-module in the list presented in \cite{Talboom}. Let $V$ be a natural $sl_{n}$-module. The $k$-fold wedge product $\wedge^kV$ is a highest weight $sl_{n}$-module. By virtue of $(E_{i,i}-E_{i+1, i+1})(e_1\wedge \cdots \wedge e_k)= \delta_{i,k}e_1\wedge \cdots \wedge e_k$ we deduce  that $\wedge^kV$ has the highest weight $\omega_k$, that is, $\wedge^kV \cong V(\omega_k)$ for $1\leq k \leq n-1$, and
$\wedge^0V=\wedge^{n}V\cong V(0)$. For convenience, we set
$\omega_0=\omega_{n}=0$, so that $V(0)=V(\omega_0)=V(\omega_{n}).$

In \cite{Rao} some $W_n$-submodules of $F^{\alpha}(\omega_k) 
$
via action
\begin{equation}\label{eq7.55}
\begin{array}{llll}
D(u,r)*(t^s\otimes v_1\wedge \dots \wedge v_k)&=&
(u | s+\alpha)t^{s+r}\otimes v_1\wedge \cdots \wedge v_k\\[3mm]
&+&\sum\limits_{p=1}^{k}(u|v_p)t^{s+r}\otimes
v_1\wedge \cdots \wedge v_{p-1}\wedge r \wedge v_{p+1} \wedge \cdots
\wedge v_k\\[3mm]
\end{array}
\end{equation}
are identified.

As $S_n$-modules they are determined in \cite{Talboom} as follows:
$$W_k^{\alpha}=\bigoplus\limits_{s\in \mathbb{Z}^{n}}t^s\otimes \mathbb{C}(s+\alpha)\wedge (\wedge^{k-1}V).$$

Then the action \eqref{eq7.55} of basis elements of $S_n$  on $W_k^{\alpha}$ is
%

\begin{equation}\label{eq7.66}
\begin{array}{lllll}
d_{i}&*&(t^s\otimes (s+\alpha)\wedge v_1 \wedge  \dots \wedge v_{k-1})\\[3mm]
&=&(s_i+\alpha_i)t^{s}\otimes (s+\alpha) \wedge v_1\wedge \cdots \wedge v_{k-1},\\[3mm]
d_{ij}(r)&*&(t^s\otimes (s+\alpha)\wedge v_1 \wedge  \dots \wedge v_{k-1})\\[3mm]
&=&\Big(r_j(s_i+\alpha_i)-r_i(s_j+\alpha_j)\Big)t^{s+r}\otimes (s+\alpha+r) \wedge v_1\wedge \cdots \wedge v_{k-1}\\[3mm]
&+&\sum\limits_{p=1}^{k-1}(r_jv_{p,i}-r_iv_{p,j})t^{s+r}\otimes (s+\alpha) \wedge v_1\wedge \cdots \wedge v_{p-1}\wedge r \wedge v_{p+1} \wedge \cdots \wedge v_{k-1}.\\[3mm]
\end{array}
\end{equation}

In case of $\alpha \in \mathbb{Z}^n$, there is another submodule of
$F^{\alpha}(\omega_k)
$. It is
$\widetilde{W}_k^{\alpha}=W_k^{\alpha} \oplus t^{-\alpha} \otimes \wedge^kV.$

The following result gives the list of irreducible cuspidal $S_n$-modules.

\begin{thm} \label{Talboom} \cite{Talboom} Let $n\geq 2$ and $0 \leq k \leq n-1$.

(a) If $\lambda \neq 0, \omega_k$, then $F^{\alpha}(\lambda)$ is irreducible.

(b) $W_k^{\alpha}$ and $F^{\alpha}(\omega_k)/ \widetilde{W}_k^{\alpha}$ are irreducible and $F^{\alpha}(\omega_k)/ \widetilde{W}_k^{\alpha}\cong W_{k+1}^{\alpha}.$
\end{thm}
Clearly, all irreducible $S_n$-modules in Theorem \ref{Talboom} are cuspidal. In fact, for $0\leq k \leq n-1$, the linear map $\psi_k \ : \ F^{\alpha}(\omega_k) \rightarrow  F^{\alpha}(\omega_{k+1})$ deinfed by
\begin{equation}\label{eqOmega}
t^s\otimes v_1\wedge \cdots \wedge v_k \xrightarrow[]{\psi_k} t^s\otimes (s+\alpha)\wedge v_1 \wedge \cdots \wedge v_k
\end{equation} is a $S_n$-module homomorphism, where $Ker(\psi_k)=\widetilde{W}_k^{\alpha}$ and $Im (\psi_k)={W}_{k+1}^{\alpha}$.

In order to apply the same approach for both $W_{n-1}$ and $S_n$-modules in verifying the Balanced Equality we will unify the module structures in Theorems \ref{mainthm} and \ref{Talboom}.

For $\lambda \neq 0, \omega_k$ the module $F^{\alpha}(\lambda)$ can be associated with $T(U, \beta)$ in the following way:
$$T(U, \alpha)=t^{\alpha}F^{\alpha}(\lambda), \quad \mbox{assuming} \quad
U=V(\lambda).$$

%
Then the action \eqref{action1} of $W_n$ on $T(U, \beta)$ is agreed with the action \eqref{eq7.33} written as
\begin{equation}\label{eq7.44}
\begin{array}{lllllll}
{d_i}*(t^s\otimes v)&=&d_i(t^{s})\otimes v,\\[3mm]
d_{ij}(r)*(t^s\otimes v)&=& \Big(d_j(t^r)d_i(t^s)-d_i(t^r)d_j(t^s)\Big)\otimes v\\[3mm]
&+&t^{s-r}\sum\limits_{p=1, p\neq i}^{n}d_p(t^r)d_j(t^r)\otimes E_{p,i}v\\[3mm]
&-&t^{s-r}\sum\limits_{p=1, p\neq j}^{n}d_p(t^r)d_i(t^r)\otimes E_{p,j}v\\[3mm]
&+&t^{s-r}d_i(t^r)d_j(t^r)\otimes(E_{i,i}-E_{j,j})v,\\[3mm]
\end{array}
\end{equation}
where $s \in \alpha + \mathbb{Z}^{n}$ and $d_j(t^s)=s_jt^s$.

Similarly, with $t^{\alpha}F^{\alpha}(\omega_k)$ one can associate $\Omega^k(\alpha).$ Then due to
$$\begin{array}{llll}
t^{\alpha}\psi_k(t^{s}\otimes v_1\wedge \cdots \wedge v_k)
&=&t^{s+\alpha}\otimes (s+\alpha)\wedge v_1 \wedge \cdots \wedge v_k\\[3mm]
&=&
\sum\limits_{i=1}^{n}(s_i+\alpha_i)t^{{s+\alpha}}\otimes e_i\wedge v_1 \wedge \cdots \wedge v_k\\[3mm]
&=&\sum\limits_{i=1}^{n}d_i(t^{s+\alpha})\otimes e_i\wedge v_1 \wedge \cdots \wedge v_k,
\end{array}$$
one can conclude that the map $t^{\alpha}\psi_k$ acts on $t^{\alpha}F^{\alpha}(\omega_k)$ as de Rham differential $d$. Therefore, $S_n$-module $W_k^{\alpha}$ can be associated with  $d(\Omega^k(\alpha))$ assuming $d(\Omega^k(\alpha))=t^{\alpha}W_k^{\alpha},$ with $0\leq k \leq n-1.$

To summarize the above arguments, we present Theorem \ref{Talboom} in a form unified with Theorem \ref{mainthm} as follows:

\begin{thm} \label{Talboom1}
\

 Let $n\geq 2$ and $0 \leq s \leq n-1$. Then the following cuspidal $S_n$-modules:

(a) a module of tensor fields $T(U, \alpha)$, where $\alpha\in \mathbb{C}^{n}$ and $U$ is a finite-dimensional
simple $sl_{n}$-module different from the $k$-th exterior power of the natural $n$-dimensional $sl_{n}$-module;

(b) a module $d(\Omega^s(\alpha))
;$

(c) a trivial 1-dimensional $S_{n}$-module,

are irreducible.
\end{thm}

\begin{rem} \label{rem7.3} From Theorem \ref{Talboom1} we conclude that irreducible cuspidal $W_n$-modules are also irreducible cuspidal $S_n$-modules with the action \eqref{eq7.2}, which is adapted version for $S_n$ of the action \eqref{action1}.
\end{rem}

We have $\widetilde{\cal S}^n\cong InDer(\cal S^n)$ and from \cite{Kac} we get $InDer(S^n)\cong S_n$, where the isomorphism is given explicitly as
$$ad_{f_1\wedge \cdots \wedge f_{n-1}} \rightarrow \sum\limits_{k=1}^{n}(-1)^{n+k}Jac^k_{f_1, \dots, f_{n-1}}d_k.$$

\

\section{Simplification of the Balanced Equality for $W_{n-1}$ and $S_n$-modules.}\label{Sec6}

\

In this section we simplify the Balanced Equality for irreducible cuspidal $W_{n-1}$ and $S_n$-modules described in Section \ref{Sec5}.

\

\subsection{Simplification of the Balanced Equality for $W_{n-1}$-modules.}

\

\

To achieve our goal, we will use the following auxiliary lemmas and relations.

\begin{lem} For any $x_i, y_i\in \cal A_{\mathbb{C}}$ we have
\begin{equation}\label{eq5.1}
\sum\limits_{p=1}^{n}(-1)^{n+p}\Big(\{y_1,\dots,\widehat{y_p},\dots,y_n\}\cdot Wr^k_{x_1,\dots,x_{n-2},y_p}+
\{x_1,\dots,x_{n-2},y_p\}\cdot Wr^k_{y_1,\dots,\widehat{y_p},\dots,y_n}\Big)=0.
\end{equation}
\end{lem}
\begin{proof} Applying Laplace's theorem we deduce
$$\sum\limits_{p=1}^{n}(-1)^{p+1}\{y_1,\dots,\widehat{y_p},\dots,y_n\}\cdot Wr^k_{x_1,\dots,x_{n-2},y_p}=
\left|\begin{array}{cccccccccc}
A & B\\
0&C\\
\end{array}\right|$$
where $A, B$ and $C$ are matrices as in $Wr^k_{x_1, \dots, x_{n-2}}, Wr^k_{y_1, \dots, y_{n}}$ and $\{y_1, \dots, y_n\},$ respectively.

Now using the properties of determinants, the upper right corner block matrix $B$ can be transformed to the following form: the first row remains unchanged, while the rest rows are equal to zero. Then again apply Laplace's theorem we get

$$\left|\begin{array}{cccccccccc}
A & B\\
0&C\\
\end{array}\right|=(-1)^{n}Jac^q(x_1, \dots, x_{n-2})[y_1, \dots, y_{n}],$$

where $Jac^q(x_1, \dots, x_{n-2})$ is the determinant which obtained from $\{x_1, \dots, x_{n-1}\}$  (see  Example \ref{exam1.2}) by deleting its $q$-th row and the last column.

In a similar way, we derive
$$\sum\limits_{p=1}^{n}(-1)^{p+1}\{x_1,\dots,x_{n-2},y_p\}\cdot Wr^k_{y_1,\dots,\widehat{y_p},\dots,y_n}=
\left|\begin{array}{cccccccccc}
A'&B'\\
0& C'\\
\end{array}\right|$$
where $A', B'$ and $C'$ are matrices as in $\{x_1, \dots, x_{n-2}\}, Wr^{n-1}_{y_1, \dots, y_{n}}$ and $Wr^k_{y_1, \dots, y_n},$ respectively.
By virtue of
$$\left|\begin{array}{cccccccccc}
A'&B'\\
0& C'\\
\end{array}\right|=(-1)^{n-1}Jac^q(x_1, \dots, x_{n-2})[y_1, \dots, y_{n}]$$
we complete the proof of lemma. \end{proof}

\begin{lem} For any $x_i, y_i\in \cal A_{\mathbb{C}}$ we have
\begin{equation}\label{eq5.3}
(-1)^{n}\{x_1, \dots, x_{n-2}, [y_1, \dots, y_{n}]\}=
\sum\limits_{p=1}^{n}(-1)^{p}[y_1, \dots,\hat{y}_p,\dots, y_{n},\{x_1, \dots, x_{n-2},y_p\}]
\end{equation}
\end{lem}
\begin{proof} The proof of lemma follows from the following chain of equalities:
$$\begin{array}{llll}
(-1)^{n}\{x_1, \dots, x_{n-2}, [y_1, \dots, y_{n}]\}&=&(-1)^{n}(-1)^{n+1}[x_1, \dots, x_{n-2}, [y_1, \dots, y_{n}],1]&=\\[3mm]
\sum\limits_{k=1}^{n}[[y_1, \dots,[x_1, \dots, x_{n-2},1,y_p],\dots, y_{n}]&=&
\sum\limits_{p=1}^{n}(-1)^{n-p+1}[y_1, \dots,\hat{y}_p,\dots, y_{n},[x_1, \dots, x_{n-2},y_p,1]]&=\\[3mm]
&&\sum\limits_{p=1}^{n}(-1)^{p}[y_1, \dots,\hat{y}_p,\dots, y_{n},\{x_1, \dots, x_{n-2},y_p\}].&
\end{array}
$$\end{proof}

Taking into account the equality
$$[f_1,\dots,f_{n-1},t^s]=\sum\limits_{k=1}^{n-1}(-1)^{n+1-k}Wr^k_{f_1, \dots, f_{n-1}}d_k(t^{s})+(-1)^{n+1}t^s\{f_1, \dots, f_{n-1}\},$$
we conclude
$$\begin{array}{llll}
ad_{f_1\wedge\dots \wedge f_{n-1}}(t^s\otimes u) &=&
\Big([{f_1,\dots,  f_{n-1},}t^{s}]
+(-1)^{n}t^s\{f_1, \dots, f_{n-1}\}\Big)\otimes u\\[3mm]
&+& t^s\sum\limits_{i,k=1}^{n-1}(-1)^{n+1-k}d_i(Wr^k_{f_1, \dots, f_{n-1}})\otimes E_{i,k}u.\\[3mm]
\end{array}$$

Therefore, the Balanced Equality
with $m=t^s\otimes u$ can be written as follows:

\begin{equation}\label{eqbig}
\begin{array}{ccccccc}
[x_1, \dots, x_{n-2}, [y_1, \dots, y_{n}],t^s]\otimes u &+&\\[3mm]
(-1)^{n}t^s\{x_1, \dots, x_{n-2}, [y_1, \dots, y_{n}]\}\otimes u&+&\\[3mm]
t^s\sum\limits_{i,k=1}^{n-1}(-1)^{n+1-k}d_i(Wr^k_{x_1, \dots, x_{n-2}, [y_1, \dots, y_{n}]}) \otimes E_{i,k}u&=&\\[3mm]
\sum\limits_{p=1}^{n}\Bigg((-1)^{n+p}[y_1, \dots, \widehat{y}_{p}, \dots, y_n,[x_1, \dots, x_{n-2}, y_p, t^s]]\otimes u&+&\\[3mm]
(-1)^{p}\underbrace{[x_1, \dots, x_{n-2}, y_p, t^s]\{y_1, \dots, \widehat{y}_{p}, \dots, y_n\}}_{1}\otimes u&+&\\[3mm]
\underbrace{[x_1, \dots, x_{n-2}, y_p,t^s] \cdot \sum\limits_{i,k=1}^{n-1}(-1)^{p+1-k}d_i(Wr^k_{y_1, \dots, \widehat{y}_p, \dots, y_n})}_2\otimes E_{i,k}u&+&\\[3mm]
(-1)^{p}\Big(\underbrace{[y_1, \dots, \widehat{y}_{p}, \dots, y_n,t^s\{x_1, \dots, x_{n-2}, y_p\}]}_3\otimes u&+&\\[3mm]
(-1)^{n}t^s\{x_1, \dots, x_{n-2}, y_p\}\{y_1, \dots, \widehat{y}_{p}, \dots, y_n\}\otimes u&+&\\[3mm]
t^s\{x_1, \dots, x_{n-2}, y_p\}\cdot
\sum\limits_{i,k=1}^{n-1}(-1)^{n+1-k}d_i(Wr^k_{y_1, \dots, \widehat{y}_p, \dots, y_n})\otimes E_{i,k}u\Big)&+&\\[3mm]
{\sum\limits_{i,k=1}^{n-1}(-1)^{p+1-k}\Big(
\underbrace{[y_1, \dots, \widehat{y}_{p}, \dots, y_n,t^sd_i(Wr^k_{x_1, \dots, x_{n-2}, y_p})]}_4\otimes E_{i,k}u}&+&\\[3mm]
(-1)^{n}t^sd_i(Wr^k_{x_1, \dots, x_{n-2}, y_p})\{y_1, \dots, \widehat{y}_{p}, \dots, y_n\}\otimes E_{i,k}u&+&\\[3mm]
{t^sd_i(Wr^k_{x_1, \dots, x_{n-2}, y_p})\cdot \sum\limits_{j,q=1}^{n-1}(-1)^{n+1-j}d_q(Wr^j_{y_1, \dots, \widehat{y}_p, \dots, y_n})\otimes E_{q,j}E_{i,k}u}
\Big)\Bigg)\\[3mm]
\end{array}
\end{equation}

From $n$-Lie identity we have
\begin{equation}\label{eq5.5}
[x_1, \dots, x_{n-2}, [y_1, \dots, y_{n}],t^s]=
\sum_{p=1}^{n}(-1)^{n+p}[y_1, \dots, \widehat{y}_{p}, \dots, y_n,[x_1, \dots, x_{n-2}, y_p, t^s]].
\end{equation}

Expanding the determinants along the last column and utilizing \eqref{Wronskianproperty}, we obtain
$$
\begin{array}{llllll}
1=\Big(\sum\limits_{k=1}^{n-1}(-1)^{n+1-k}Wr^k_{x_1, \dots, x_{n-2},y_p}d_k(t^{s})+(-1)^{n+1}t^s\{x_1, \dots, x_{n-2}, y_p\}\Big)\cdot\{y_1, \dots, \widehat{y}_{p}, \dots, y_n\}, \\[1mm]
2=
\Big(\sum\limits_{q=1}^{n-1}(-1)^{n-q}Wr^q_{x_1, \dots, x_{n-2},y_p}d_q(t^{s})+(-1)^{n}t^s\{x_1, \dots, x_{n-2}, y_p\}\Big)\cdot \sum\limits_{i,k=1}^{n-1}(-1)^{p-k}d_i(Wr^k_{y_1, \dots, \widehat{y}_p, \dots, y_n}), \\[1mm]
3=
t^s[y_1, \dots, \widehat{y}_{p}, \dots, y_n,\{x_1, \dots, x_{n-2}, y_p\}]+
\{x_1, \dots, x_{n-2}, y_p\}\cdot
\sum\limits_{k=1}^{n-1}(-1)^{n+1-k}Wr^k_{y_1, \dots, \widehat{y}_{p}, \dots, y_n}d_k(t^{s}),\\[1mm]
4=
%
%
t^s[y_1, \dots, \widehat{y}_{p}, \dots, y_n,d_i(Wr^k_{x_1, \dots, x_{n-2}, y_p})]+d_i(Wr^k_{x_1, \dots, x_{n-2}, y_p})\cdot
\sum\limits_{q=1}^{n-1}(-1)^{n+1-q}Wr^q_{y_1, \dots, \widehat{y}_{p}, \dots, y_n}d_q(t^{s}).\\[1mm]
\end{array}
$$

Substituting the above expressions into the Balanced Equality and taking into account the equalities \eqref{eq5.1}, \eqref{eq5.3}, \eqref{eq5.5}, we group similar terms and, after dividing both sides of the Balanced Equality by $t^s$, we obtain

$$\begin{array}{cccc}
\sum\limits_{i,k=1}^{n-1}(-1)^{n+1-k}d_i(Wr^k_{x_1, \dots, x_{n-2}, [y_1, \dots, y_{n}]}) \otimes E_{i,k}u&=&\\[5mm]
\sum\limits_{p=1}^{n}\Bigg(\sum\limits_{q=1}^{n-1}(-1)^{n+1-q}{s_q}Wr^q_{x_1, \dots, x_{n-2},y_p}\cdot \sum\limits_{i,k=1}^{n-1}(-1)^{p+1-k}d_i(Wr^k_{y_1, \dots, \widehat{y}_p, \dots, y_n})\otimes E_{i,k}u&+&\\[5mm]
\sum\limits_{i,k=1}^{n-1}(-1)^{p+1-k}\Bigg(\Big([y_1, \dots, \widehat{y}_{p}, \dots, y_n,d_i(Wr^k_{x_1, \dots, x_{n-2}, y_p})]&+&\\[5mm]
d_i(Wr^k_{x_1, \dots, x_{n-2}, y_p})\cdot
\sum\limits_{q=1}^{n-1}(-1)^{n+1-q}{s_q}Wr^q_{y_1, \dots, \widehat{y}_{p}, \dots, y_n}\Big)\otimes E_{i,k}u&+&\\[5mm]
(-1)^{n}d_i(Wr^k_{x_1, \dots, x_{n-2}, y_p})\{y_1, \dots, \widehat{y}_{p}, \dots, y_n\}\otimes E_{i,k}u&+&\\[5mm]
d_i(Wr^k_{x_1, \dots, x_{n-2}, y_p})\cdot \sum\limits_{j,q=1}^{n-1}(-1)^{n+1-j}d_q(Wr^j_{y_1, \dots, \widehat{y}_p, \dots, y_n})\otimes E_{q,j}E_{i,k}u
\Bigg)\Bigg).\\[5mm]
\end{array}
$$

Due to arbitrariness of $s$ (i.e., all coordinates of $s$ are arbitrary) the above equality is reduced into the following two equalities:

\begin{equation}\label{eq5.6}
\begin{array}{ccccc}
\sum\limits_{i,k=1}^{n-1}(-1)^{k+1}\Bigg( \sum\limits_{p=1}^{n}(-1)^{p}\Big(Wr^q_{x_1, \dots, x_{n-2},y_p}d_i(Wr^k_{y_1, \dots, \widehat{y}_p, \dots, y_n})&+&\\[3mm]
d_i(Wr^k_{x_1, \dots, x_{n-2}, y_p})Wr^q_{y_1, \dots, \widehat{y}_{p}, \dots, y_n}\Big)\Bigg)\otimes E_{i,k}u&=&0,\\[3mm]
\end{array}
\end{equation}
for any $1\leq q \leq n-1$.

\begin{equation}\label{eq5.7}
\begin{array}{ccccccc}
\sum\limits_{i,k=1}^{n-1}(-1)^{n+1-k}d_i(Wr^k_{x_1, \dots, x_{n-2}, [y_1, \dots, y_{n}]}) \otimes E_{i,k}u&+&\\[5mm]
\sum\limits_{p=1}^{n}\sum\limits_{i,k=1}^{n-1}(-1)^{p-k}\Bigg(\Big([y_1, \dots, \widehat{y}_{p}, \dots, y_n,d_i(Wr^k_{x_1, \dots, x_{n-2}, y_p})]&+&\\[5mm]
(-1)^{n+1}d_i(Wr^k_{x_1, \dots, x_{n-2}, y_p})\{y_1, \dots, \widehat{y}_{p}, \dots, y_n\}\otimes E_{i,k}u&+&\\[5mm]
d_i(Wr^k_{x_1, \dots, x_{n-2}, y_p})\cdot \sum\limits_{j,q=1}^{n-1}(-1)^{n-j}d_q(Wr^j_{y_1, \dots, \widehat{y}_p, \dots, y_n})\otimes E_{q,j}E_{i,k}u\Bigg)&=&0.\\[5mm]
\end{array}\end{equation}

\

\subsection{Simplification of the Balanced Equality for $S_{n}$-modules.}

\

\

Since
$\{f_1, \dots, f_{n}\}=Jac_{f_1, \dots, f_{n}}=Wr_{1, f_1, \dots, f_{n}}=[1, f_1, \dots, f_{n}],$ we obtain the embedding $S^n\subset W^{n+1}$.

By virtue of the following equality
$$\sum\limits_{i=1}^{n}(-1)^{n+k}Jac^k_{f_1, \dots, f_{n-1}}d_k=
\sum_{k=1}^{n}(-1)^{n+k}Wr^k_{1, f_1, \dots, f_{n-1}}d_k$$
and Remark \ref{rem7.3} from Equalities \eqref{eq5.6} and \eqref{eq5.7} we
easily reduce the Balances Equality \eqref{eq3.7} for $S_{n}$-modules to the following ones:

\begin{equation}\label{eq7.88}
\begin{array}{lllll}
\sum\limits_{p, i, k =1}^{n} (-1)^{k+p}\Bigg(Jac^q_{x_1, \dots, x_{n-2}, y_p}d_i(Jac^k_{y_1, \dots, \widehat{y}_p, \dots, y_n})&+&\\[5mm]
(Jac^k_{x_1, \dots, x_{n-2}, y_p})
Jac^q_{y_1, \dots, \widehat{y}_p, \dots, y_n}\Bigg)\otimes E_{i,k}u&=&0,\\[5mm]
\end{array}
\end{equation}
for any $1\leq q \leq n$.

\begin{equation}\label{eq7.99} \begin{array}{llllll}
\sum\limits_{i,k=1}^{n}(-1)^{n+k}
d_i(Jac^k_{x_1, \dots, x_{n-2}, \{y_1, \dots, y_{n}\}})\otimes E_{i,k}u&=&\\[3mm]
\sum\limits_{i,k=1}^{n}(-1)^{k}\sum\limits_{p=1}^{n}(-1)^{p}\Bigg(
\{y_1, \dots, \widehat{y}_p, \dots, y_n,d_i(Jac^k_{x_1, \dots, x_{n-2}, y_p})\}\otimes E_{i,k}u&+&\\[5mm]
d_i(Jac^k_{x_1, \dots, x_{n-2}, y_p})\sum\limits_{j,q=1}^{n}(-1)^{n+j}
d_q(Jac^j_{y_1, \dots, \widehat{y}_p, \dots, y_n})\otimes E_{q,j}E_{i,k}u\Bigg).\\[5mm]
\end{array}
\end{equation}

\

\section{Irreducible cuspidal modules over $n$-Lie algebra $W^{n}$.}

\

Thanks to results of Section \ref{Sec6} and Proposition \ref{prop3.3} in this section we are going to examine cuspidal $W_{n-1}$-modules listed in Theorem \ref{mainthm} that satisfy Equalities \eqref{eq5.6} - \eqref{eq5.7}.

\

\subsection{The case of $T(U, \alpha)$ with $dim U=1$.}\label{subsection 7.1}

\

\

It is known that $gl_{n-1}=sl_{n-1}\oplus \mathbb{C} E$, where $E$ is unitary matrix and any $gl_{n-1}$-module $U$ is an extension of $sl_{n-1}$-module $V$ with $E*V=b\cdot id_V$. Consider the standard basis (Chevalley's basis) of $sl_{n-1}$:
$$h_k=E_{k,k}-E_{n-1,n-1}, \quad 1\leq k \leq n-2, \quad E_{i,j}, \quad 1\leq i\neq j \leq n-1.$$

Let us the first consider the case of $U=Span\{u_0\}$. In this case, it is a trivial module over $sl_{n-1}$. Consequently, we have
\begin{equation}\label{eq6.1}
E_{i,k}u_0=0, \quad E_{k,k}u_0=\frac{b}{n-1}u_0, \quad E_{q,j}E_{i,k}u_0=\delta_{i,k}\delta_{q,j}(\frac{b}{n-1})^2u_0, \ 1\le i,j,q,k\le n-1.
\end{equation}

Denote by
$$W^q(x), \quad W^q(y), \quad J^q(x), \quad W(y), \quad J^q(y), \quad J(y), \quad J(x), \quad W^q(y)$$
the matrices (not necessarily  square) as in
$$Wr^q_{x_1, \dots, x_{n-2}}, \ Wr^q_{y_1, \dots, y_{n}}, \ Jac^q(x_1, \dots, x_{n-2}), \ Wr_{y_1, \dots, y_n},$$
$$ \ Jac^q(y_1, \dots, y_n), \ Jac(y_1, \dots, y_n), \ Jac(x_1, \dots, x_{n-2}), \ Wr^q_{y_1, \dots, y_n},$$
respectively.

We set also by $J^{q=0}(x), \ J^{q=0}(y)$ the  matrices corresponding to $Jac_0^q(x_1, \dots, x_{n-2}), \ Jac_0^q(y_1, \dots, y_n)$ with the $q$-th row set to zero, respectively and by $J^{q,0}(x), \ J^{q,0}(y)$ are matrices corresponding to $Jac^{q}(x_1, \dots, x_{n-2}), \ Jac^{q}(y_1, \dots, y_n)$ with the first row set to zero, respectively.

\begin{prop} \label{prop6.1} The space $$T(U, \alpha)=t^{\alpha}\mathbb{C}[t_1^{\pm 1}, \dots, t_{n-1}^{\pm 1}]\otimes U \quad \mbox{with} \quad dim U=1$$
is simple cuspidal module over $W^n$.
\end{prop}
\begin{proof} We prove the proposition by straightforward verifying of
Equalities \eqref{eq5.6} and \eqref{eq5.7}.

First, we note that by using the action of a derivation on a determinant, one can obtain
\begin{equation}\label{eq5.2}
\sum\limits_{k=1}^{n-1}(-1)^kd_k(Wr^k_{y_1,\dots, \widehat{y}_p, \dots, y_n})=-(n-1)\{y_1,\dots, \widehat{y}_p, \dots, y_n\}.
\end{equation}

Substituting now \eqref{eq5.2} along with \eqref{eq6.1} into \eqref{eq5.6} we derive

%
%

$$\sum\limits_{p=1}^{n}(-1)^{p}\Big(Wr^q_{x_1, \dots, x_{n-2},y_p} \{y_1, \dots, \widehat{y}_p, \dots, y_n\}+\{x_1, \dots, x_{n-2}, y_p\}
Wr^q_{y_1, \dots, \widehat{y}_{p}, \dots, y_n}\Big)
=$$

%
%
%
%

$$
-\left|\begin{array}{cccccccc}
W^q(x)&W^q(y)\\
0&J(y)\end{array}\right|
-\left|
\begin{array}{cccccccc}
J(x)&J(y)\\
0&W^q(y)\\
\end{array}\right|
=
-\left|\begin{array}{cccccccc}
J^{q,0}(x)& W^q(y)\\
0&J(y)\\
\end{array}\right|-
\left|\begin{array}{cccccccc}
W^q(x)&J^{q,0}(y)\\
0&J(y)\\
\end{array}\right|
-$$

$$
\left|\begin{array}{cccccccc}
J^{q=0}(x)&J(y)\\
0&W^q(y)\\
\end{array}\right|-
\left|\begin{array}{cccccccc}
J(x)&J^{q=0}(y)\\
0&W^q(y)\\
\end{array}\right|
=
(-1)^{n-1}\left|\begin{array}{cccccccc}
J^q(x)&J^q(y)\\
0&W(y)\\
\end{array}\right|-
\left|\begin{array}{cccccccc}
W^q(x)&J^{q,0}(y)\\
0&J(y)\\
\end{array}\right|-
$$
$$
-\left|\begin{array}{cccccccc}
J^{q,0}(x)&J(y)\\
0&W^q(y)\\
\end{array}\right|-
\left|\begin{array}{cccccccc}
J(x)&J^{q,0}(y)\\
0&W^q(y)\\
\end{array}\right|
=
-
\left|\begin{array}{cccccccc}
W^q(x)&0\\
0&J(y)\\
\end{array}\right|-
\left|\begin{array}{cccccccc}
J(x)&0\\
0&W^q(y)\\
\end{array}\right|
=0,$$

It implies that Equality \eqref{eq5.6} holds true for $dim U=1.$

Next, expanding the first term of Equality \eqref{eq5.7} along the last column and grouping similar terms we derive

$$\begin{array}{ccc}
\underbrace{\sum\limits_{k=1}^{n-1}\Bigg(\sum\limits_{p=1}^{n}
\sum\limits_{q=1}^{n-1}(-1)^{n+k}\Big((-1)^{p+q}d_q(d_k(Wr^k_{x_1, \dots, x_{n-2}, y_p}))Wr^q_{y_1, \dots, \widehat{y}_{p}, \dots, y_n}}_{(a)}+\\[3mm]
\underbrace{d_k(Wr^k_{x_1, \dots, x_{n-2}, [y_1, \dots, y_{n}]})\Big)
\Bigg)}_{(a)}\otimes \frac{b}{n-1}u_0+
\end{array}$$

$$\underbrace{\sum\limits_{p=1}^{n}(-1)^{p+n}\sum\limits_{k=1}^{n-1}(-1)^{k}
d_k(Wr^k_{x_1, \dots, x_{n-2}, y_p})\cdot \sum\limits_{q=1}^{n-1}(-1)^{q}d_q(Wr^q_{y_1, \dots, \widehat{y}_p, \dots, y_n})}_{(b)}\otimes (\frac{b}{n-1})^2u_0=0.$$

Using \eqref{eq5.2} and then expand $\{x_1, \dots, x_{n-2}, y_p\}$ along the last column we obtain

$$(a)=
(n-1)(-1)^{n-1}\Bigg(\sum\limits_{p=1}^{n}\sum\limits_{q=1}^{n-1}(-1)^{p+q}d_q\Big(\sum\limits_{i=1}^{n-1}(-1)^{n-1+i}d_i(y_p)Jac^i(x_1, \dots, x_{n-2})\Big)Wr^q_{y_1, \dots, \widehat{y}_{p}, \dots, y_n}+$$
$$\{x_1, \dots, x_{n-2}, [y_1, \dots, y_{n}]\}\Bigg)\overset{derivation \ property \ for \ d_q}{=}$$

$$(n-1)(-1)^{n-1}\Bigg(\sum\limits_{i=1}^{n-1}(-1)^{n+i}\sum\limits_{p=1}^{n}\sum\limits_{q=1}^{n-1}(-1)^{p+q+1}d_i(d_q(y_p))Wr^q_{y_1, \dots, \widehat{y}_{p}, \dots, y_n}Jac^i(x_1, \dots, x_{n-2})+$$
$$\sum\limits_{q,i=1}^{n-1}(-1)^{n+i}\sum\limits_{p=1}^{n}(-1)^{q+p+1}d_i(y_p)Wr^q_{y_1, \dots, \widehat{y}_{p}, \dots, y_n}d_q\Big(Jac^i(x_1, \dots, x_{n-2})\Big)+\{x_1, \dots, x_{n-2}, [y_1, \dots, y_{n}]\}\Bigg)=$$

$$(n-1)(-1)^{n-1}\Bigg(\sum\limits_{i=1}^{n-1}(-1)^{n+i}d_i[y_1, \dots,  y_n]Jac^i(x_1, \dots, x_{n-2})+$$
$$\sum\limits_{q,i=1}^{n-1}(-1)^{n+i}\sum\limits_{p=1}^{n}(-1)^{q+p+1}d_i(y_p)Wr^q_{y_1, \dots, \widehat{y}_{p}, \dots, y_n}d_q\Big(Jac^i(x_1, \dots, x_{n-2})\Big)+\{x_1, \dots, x_{n-2}, [y_1, \dots, y_{n}]\}\Bigg)=$$



$$(n-1)\sum\limits_{q,i=1}^{n-1}(-1)^{i} \sum\limits_{p=1}^{n}(-1)^{p+q}d_i(y_p)Wr^q_{y_1, \dots, \widehat{y}_{p}, \dots, y_n}d_q\Big(Jac^i(x_1, \dots, x_{n-2})\Big)=$$

$$(n-1)\sum\limits_{q,i=1}^{n-1}(-1)^{i+1} det(A)d_q\Big(Jac^i(x_1, \dots, x_{n-2})\Big)=
(1-n)[y_1,\dots,y_n]\sum\limits_{q=1}^{n-1}(-1)^{q}d_q\Big(Jac^q(x_1, \dots, x_{n-2})\Big)=0,$$
where $det(A)$ is obtained from $[y_1, \dots, y_n]$ by replacing $d_q$ to $d_i.$
%
%
%
%
%
%

Again applying \eqref{eq5.2} we derive

$$(b)=
(n-1)^2\sum\limits_{p=1}^{n}(-1)^{p+n}\{x_1, \dots, x_{n-2}, y_p\}\{y_1, \dots, \widehat{y}_p, \dots, y_n\}=
(n-1)^2(-1)^{n+1}\left|\begin{array}{ccccccc}
J(x)&J(y)\\
0&J(y)\\
\end{array}
\right|=
0.$$
The proof is completed.
\end{proof}

\

\subsection{The case of $T(U, \alpha)$ with $dim U > 1$.}

\

\

Let us rewrite Equality \eqref{eq5.6} in the following form
\begin{equation}\label{eq6.5}\sum\limits_{i,k=1}^{n-1}(-1)^{k+1}a_{i,k}\otimes E_{i,k}u=0, \quad \mbox{for any} \quad u\in U,
\end{equation}
where
$$a_{i,k}=\sum\limits_{p=1}^{n}(-1)^{p}\Big(Wr^q_{x_1, \dots, x_{n-2},y_p}d_i(Wr^k_{y_1, \dots, \widehat{y}_p, \dots, y_n})+
d_i(Wr^k_{x_1, \dots, x_{n-2}, y_p})
Wr^q_{y_1, \dots, \widehat{y}_{p}, \dots, y_n}\Big).$$

Since $U$ is a simple $gl_{n-1}$-module, it is the highest weight module. Consider the above equality for $u_{\lambda}$ (here $u_{\lambda}$ is the highest weight vector of $U$)
\begin{equation}\label{eq6.6}\sum\limits_{1\leq k \leq i \leq n-1}(-1)^{k+1}a_{i,k}\otimes E_{i,k}u_{\lambda}=0.
\end{equation}

\begin{lem} \label{lem6.2} We have that $a_{i,1}=0$ for some $i\in \{2, \dots, n-1\}$.

\end{lem}

\begin{proof} Let us assume that $E_{i, 1}u_{\lambda}=0$ for any $i\in\{2, \dots, n-1\}.$ From
$$0=E_{1,i}(E_{i, 1}u_{\lambda})=E_{i,1}(E_{1, i}u_{\lambda})+[E_{1,i},E_{i, 1}]u_{\lambda}=[E_{1,i},E_{i, 1}]u_{\lambda}=(E_{1,1}-E_{i, i})u_{\lambda}$$
we conclude $(E_{1,1}-E_{i, i})u_{\lambda}=0$ for any $2\leq i \leq n-1$ . It implies $h_ku_{\lambda}=0$ for any $k\in\{1, \dots, n-2\}$ and hence, $dimU=1.$ This contradiction leads to existence of $i_0\in \{2, \dots, n-1\}$ such that $E_{i_0,1}u_{\lambda}\neq 0.$ Since no other element of the form
$E_{i,k}u_{\lambda}$ belongs to the same weight space as $E_{i_0,1}u_{\lambda}$ we obtain $a_{i_0,1}=0$.
\end{proof}

Below we present $x_p, y_q\in \mathbb{C}[t_1^{\pm 1}, \dots, t_{n-1}^{\pm 1}]$ with $1\leq p \leq n-2$ and $1\leq q \leq n$ such that $a_{i,1}\neq 0$ for any $i\in\{2, \dots, n-1\}$.

\begin{exam} \label{exam6.3} Let $n\geq 4$.
\begin{itemize}
\item For a fixed $i\in \{2, \dots, n-2\}$ we consider the following elements:
$$x_1=1, \quad x_j=t_j \quad \mbox{for} \quad j<i \quad x_j=t_{j+1} \quad \mbox{for} \quad  i\leq j \leq n-2,$$
$$y_j=t_j \quad \mbox{for} \quad j<i \quad y_i=t_1t_i, \quad y_{i+1}=t_i^2, \quad y_j=t_{j-2}\quad \mbox{for} \quad  i+2\leq j \leq n,$$
Taking into account that
$$\begin{array}{llll}
 Wr^1_{y_1, \dots, \widehat{y}_p, \dots, y_n}=0, \ \ 1\le p\le n\\[1mm]
 Wr^q_{y_1, \dots, \widehat{y}_{p}, \dots, y_n}=0, \ \ q\neq n-1,\\[1mm]
 Wr^1_{x_1, \dots, x_{n-2}, y_p}=0, \ \ p\neq \{i,i+1,i+2\},\end{array}$$
we obtain
$\begin{array}{ccc}
a_{i,1}= 
(-1)^{n+i}2t_i^2(t_1t_2t_3\dots t_{n-2})^2t_{n-1}\neq0 \end{array}$ 
for $2\le i\le n-2$.

\item Let us consider the following elements
$$x_1=1,\ x_j=t_j,\ 2\leq j \leq n-2,$$
$$y_j=t_j,\ 1\leq j \leq n-3, \ y_{n-2}=t_1t_{n-1},\ y_{n-1}=t_{n-1}^2,\ y_{n}=t_{n-1}.$$
Since
$$\begin{array}{llll}
Wr^1_{y_1, \dots, \widehat{y}_p, \dots, y_n}=0, \ \ 1\le p\le n\\[1mm]
Wr^q_{y_1, \dots, \widehat{y}_{p}, \dots, y_n}=0, \ \ q\neq n-2,\\[1mm]
Wr^1_{x_1, \dots, x_{n-2}, y_p}=0, \ \ p\neq \{n-2,n-1,n\},\end{array}$$
we get
$\begin{array}{llll}
a_{n-1,1}=
2(t_1t_2\dots t_{n-3})^2t_{n-1}^4t_{n-2}\neq 0.\end{array}
$
\end{itemize}

Let $n=3, q=2$ and $x_1=1,\ y_1=t_1,\ y_2=t_2,\ y_3=t_2^2$. Then $a_{2,1}=-4t_1t_2^3\neq0.$
\end{exam}

From Lemma \ref{lem6.2} and Example \ref{exam6.3} we obtain the following result.
\begin{prop} \label{prop6.4} A cuspidal module $T(U, \alpha)$ with $dim U > 1$ is not simple module over $W^n$.
\end{prop}

\

\subsection{The case of the second cuspidal $W_{n-1}$-module.}

\

\

Since
\begin{equation}\label{eqder}
W_{n-1}*(d(\Omega^s(\alpha)))=d(W_{n-1}*\Omega^s(\alpha)),
\end{equation}
we only need to check whether left sides of Equalities \eqref{eq5.6} and \eqref{eq5.7} lie in $Ker(d)$.

Applying similar arguments as those used in Subsection \ref{subsection 7.1}, we deduce that the Balanced Equality holds true for $\Omega^0(\alpha)$. Therefore, using \eqref{eqder}, we obtain that simple $W_{n-1}$-module $d(\Omega^0(\alpha))$ is also simple $W^n$-module.

For the element $u=e_1\wedge \dots \wedge e_s\in \wedge^s V$ with $1\leq s \leq n-2$ we consider the left side of \eqref{eq5.6}:

$$\sum\limits_{i,k=1}^{n-1}(-1)^{k+1}a_{i,k}\otimes E_{i,k}u=
\sum\limits_{k=1}^{s}\Big((-1)^{k-1}a_{k,k}\otimes e_1\wedge \dots \wedge e_s+
\sum\limits_{i=s+1}^{n-1}a_{i,k}\otimes e_i\wedge e_1\wedge \dots \wedge \widehat{e_k} \wedge \dots \wedge e_s\Big).$$

Now, applying the de Rham differential $d$ and using the properties of the wedge product to the expression obtained above, we obtain
$$\sum\limits_{k=1}^{s}\sum\limits_{j=1}^{n-1}\Bigg((-1)^{k-1}d_j(a_{k,k})\otimes e_j\wedge e_1\wedge \dots \wedge e_s+
\sum\limits_{i=s+1}^{n-1}d_j(a_{i,k})\otimes e_j\wedge e_i\wedge e_1\wedge \dots \wedge \widehat{e_k} \wedge \dots \wedge e_s\Bigg)=$$

%


$$\sum\limits_{k=1}^{s}\Bigg(\sum\limits_{i=s+1}^{n-1}(-1)^{k-1}
\Big(d_i(a_{k,k})-d_k(a_{i,k})\Big)\otimes e_i\wedge e_1\wedge \dots \wedge e_s+$$
$$\sum\limits_{s+1\leq i < j \leq n-1}\Big(d_j(a_{i,k})-d_i(a_{j,k})\Big)\otimes e_j\wedge e_i\wedge e_1\wedge \dots \wedge \widehat{e_k} \wedge \dots \wedge e_s\Bigg).$$



Comparing elements, we find that left side of \eqref{eq5.6} involving the element $u,$ lies in $Ker(d)$ if the following conditions hold:
\begin{equation}\label{k}
\left\{\begin{array}{lllll}
\sum\limits_{k=1}^{s}(-1)^{k-1}\Big(d_i(a_{k,k})-d_k(a_{i,k})\Big)=0, & s+1 \leq i \leq n-1, &1\leq s\leq n-2, \\[3mm]
\sum\limits_{k=1}^{s}\Big(d_j(a_{i,k})-d_i(a_{j,k})\Big)=0,& s+1\leq i < j \leq n-1.\\[3mm]
\end{array}\right.
\end{equation}

\begin{exam} \label{exam6.5}
For a fixed $s\in \{1, \dots, n-2\}$ we consider the following elements of $\cal A_{n-1}$:
$$\begin{array}{llllllll}
x_1=1, & x_j=t_j, & 2\le j\le s, & x_j=t_{j+1}, & s+1\le j\le n-2,&&\\[3mm]
y_j=t_j, & 1\le j\le s, & y_{s+1}=t_1t_{s+1}, & y_{s+2}=t_{s+1}^a, \quad y_{j}=t_{j-2}, & s+3\le j\le n.\\[3mm]
\end{array}$$

Note that
$$ Wr^q_{y_1, \dots, \widehat{y}_{p}, \dots, y_n}=0,\quad q\neq n-1 \quad \mbox{and} \quad W^k_{x_1\dots x_{n-2},y_p}=0,\quad 2\leq k\leq s.$$

After a straightforward computation for $i=s+1$, we obtain

$$\begin{array}{ccccc}
\sum\limits_{k=1}^{s}(-1)^{k-1}\Big(d_{s+1}(a_{k,k})-d_k(a_{s+1,k})\Big)=
(-1)^{s+n}\Big(3a^2-a-2\Big)(t_1t_2t_3\dots t_{n-2})^2t_{n-1}t_{s+1}^{a-1}.\\[5mm]
\end{array}$$

Therefore, the second equality in \eqref{k} is not true for the chosen elements $x_i, y_i$ with $a>1$.
\end{exam}

From Example \ref{exam6.5} we obtain the following result.
\begin{prop} \label{prop6.6}  $d(\Omega^s(\alpha))$ is simple module over $W^n$ only for $s=0$.
\end{prop}

From Propositions \ref{prop6.1}, \ref{prop6.4} and \ref{prop6.6} we obtain the main result of this work related to irreducible cuspidal $W^n$-modules.

\begin{thm} \label{thm6.7} Simple cuspidal modules over $W^n$ are the followings:
\begin{itemize}
  \item $T(U, \alpha)$ with one-dimensional trivial $gl_{n-1}$-module $U=<u>$ under the action
$$[x_1, \dots, x_{n-1}, t^s\otimes u]=
\sum\limits_{i=1}^{n-1}(-1)^{n+1-i}
\Big(Wr^i_{x_1, \dots, x_{n-1}}d_i(t^s)+\frac{b}{n-1}t^sd_i(Wr^i_{x_1, \dots, x_{n-1}})\Big)\otimes u;$$

  \item $d(\Omega^0(\alpha))$
  under the action
  $$[x_1, \dots, x_{n-1}, \sum\limits_{i=1}^{n-1}d_i(t^s)\otimes e_i]=
\sum\limits_{i=1}^{n-1}d_i\Bigg(\sum_{k=1}^{n-1}(-1)^{n+1-k}
Wr^k_{x_1, \dots, x_{n-1}}d_k(t^s)\Bigg)\otimes e_i,$$

where $\{e_1, \dots, e_{n-1}\}$ is a basis of the natural $gl_{n-1}$-module;
  \item $1$-dimensional trivial.
\end{itemize}
\end{thm}

Note that $W^n$-module $T(U, \alpha)$ in the list of Theorem \ref{thm6.7} depends on complex parameters $b, \ \alpha_i, 1\leq i \leq n-1,$ while $d(\Omega^0(\alpha))$ depends only on parameters $\alpha_i, 1\leq i \leq n-1$. Moreover, if $\alpha_i-\alpha'_i\in \mathbb{Z}^{n-1}$ for any $1\leq i \leq n-1$, then modules $T(U, \alpha)$ and $T(U, \alpha')$ are isomorphic. The same we have for modules $d(\Omega^0(\alpha))$. Note that the regular module over $W^n$ corresponds to the module
$T(\mathbb{C}, 0).$

\

\section{Irreducible cuspidal modules over $n$-Lie algebra $S^{n}$.}

\

In this section we are going to identify $S_{n}$-modules in the listed of Theorem \ref{Talboom1} that satisfy Equalities \eqref{eq7.88} - \eqref{eq7.99}. It is obvious that Equalities \eqref{eq7.88} and \eqref{eq7.99} are true for $T(U, \alpha)$ with $dimU=1$.

\

\subsection{The case of $T(U, \alpha)$ with $dimU>1$.}

\

\

Let us write Equality \eqref{eq7.88} in the following form
\begin{equation}\label{eq7.89}\sum\limits_{i,k=1}^{n}(-1)^{k}a_{i,k}\otimes E_{i,k}u=0,
\end{equation}
where
$$a_{i,k}=\sum\limits_{p=1}^{n}(-1)^{p}\Big(Jac^q_{x_1, \dots, x_{n-2}, y_p}d_i(Jac^k_{y_1, \dots, \widehat{y}_p, \dots, y_n})+
d_i(Jac^k_{x_1, \dots, x_{n-2}, y_p})
Jac^q_{y_1, \dots, \widehat{y}_p, \dots, y_n}\Big).$$

Since $U$ is a simple $sl_{n}$-module, it is the highest weight module. For the highest weight vector $u_{\lambda}$ Equality \eqref{eq7.89} has the form
\begin{equation}\label{eq6.6}\sum\limits_{1\leq k \leq i \leq n}(-1)^{k+1}a_{i,k}\otimes E_{i,k}u_{\lambda}=0.
\end{equation}

From Lemma \ref{lem6.2} we get $a_{i,1}=0$ for some $i\in \{2, \dots, n\}$. However, in next example we present elements $x_p, y_q\in \mathbb{C}[t_1^{\pm 1}, \dots, t_n^{\pm 1}]$ with $1\leq p \leq n-2$ and $1\leq q \leq n$ such that $a_{i,1}\neq 0$ for any $i\in\{2, \dots, n\}$. This proves that $S_n$-module $T(U, \alpha)$ with $dimU>1$ is not $S^n$-module for $n\geq 3$.

\begin{exam} \label{exam6.3}
\

\begin{itemize}
\item For a fixed $i\in \{2, \dots, n-1\}$ we consider the following elements:
$$x_j=t_{j+1} \quad \mbox{for} \quad j\le i-2 \quad x_j=t_{j+2} \quad \mbox{for} \quad  i-1\leq j \leq n-2,$$
$$y_j=t_j \quad \mbox{for} \quad j\le i-1 \quad y_{i}=t_i^2,\quad y_{i+1}=t_1t_i,   \quad y_j=t_{j-1}\quad \mbox{for} \quad  i+2\leq j \leq n,$$
Taking into account that
$$\begin{array}{llll}
 Jac^1_{y_1, \dots, \widehat{y}_p, \dots, y_n}=0, \  \ 1\le p\le n\\[1mm]
 Jac^q_{y_1, \dots, \widehat{y}_{p}, \dots, y_n}=0, \  \ q\neq n,\\[1mm]
 Jac^1_{x_1, \dots, x_{n-2}, y_p}=0, \  \ p\neq \{i,i+1\},\end{array}$$
we obtain
$a_{i,1}=
2(-1)^{n}(t_1t_2\dots t_{n-1})^2t_it_n \neq 0.$

\item Consider the elements
$$x_j=t_{j+1},\quad  1\leq j \leq n-2, \quad y_j=t_j, \quad  1\leq j \leq n-2, \quad y_{n-1}=t_n^2, \quad y_{n}=t_1t_{n}.$$
Since
$$\begin{array}{llll}
Jac^1_{y_1, \dots, \widehat{y}_p, \dots, y_n}=0, \ \ 1\le p\le n\\[1mm]
Jac^q_{y_1, \dots, \widehat{y}_{p}, \dots, y_n}=0, \  \ q\neq n-1,\\[1mm]
Jac^1_{x_1, \dots, x_{n-2}, y_p}=0, \ \ p\neq \{n-1,n\},\end{array}$$
we get
$a_{n,1}=
2(-1)^{n-1}(t_1t_2\dots t_{n-2})^2t_{n-1}t_n^3\neq0.$
%
%

\end{itemize}

\end{exam}

\

\subsection{Cuspidal $S_n$-modules of the type $d(\Omega^s(\alpha))$.}

\

\

Since $S_n$ is a subalgebra of $W_n$, and the modules $d(\Omega^s(\alpha))$ in both are nearly identical (differing only in their interaction with the natural modules of $gl_n$ and $sl_n$, verifying Equality \eqref{eq7.89} leads to the restrictions \eqref{k}, with replacing $(n-1)$ to $n$. Namely, we get the following
restrictions:
\begin{equation}\label{k2}
\left\{\begin{array}{lllll}
\sum\limits_{k=1}^{s}(-1)^{k-1}\Big(d_i(a_{k,k})-d_k(a_{i,k})\Big)=0, & for \ any &s+1 \leq i \leq n, &1\leq s\leq n-1, \\[5mm]
\sum\limits_{k=1}^{s}\Big(d_j(a_{i,k})-d_i(a_{j,k})\Big)=0& for \ any &s+1\leq i < j \leq n.\\[5mm]
\end{array}\right.
\end{equation}

\begin{exam} \label{exam7.5}
\

\begin{itemize}

\item For a fixed $s\in \{1, \dots, n-2\}$ and $a\in \mathbb{N}$ we consider elements of $\cal A_{n}$:
$$\begin{array}{lllll}
x_j=t_{j+1}, & 1\le j\le s-1, & x_j=t_{j+2}, \ s\le j\le n-2,\\[3mm]
y_j=t_j, & 1\le j\le s, & y_{s+1}=t_{s+1}^a,\ y_{s+2}=t_1t_{s+1},& y_{j}=t_{j-1}, & s+3\le j\le n.
\end{array}$$

Note that $Jac^l_{y_1, \dots, \widehat{y}_{p}, \dots, y_n}=0$ for $l\neq n$ and $Jac^l_{x_1\dots x_{n-2},y_p}=0$ for $2\leq l\leq s.$

Suppose $i=s+1,$ then we get
$$
\begin{array}{lllll}
\sum\limits_{k=1}^{s}(-1)^{k-1}\Big(d_{s+1}(a_{k,k})-d_k(a_{s+1,k})\Big)
=
(-1)^{n}a(a-3)(t_1t_2\dots t_{n-1})^2t_{s+1}^{a-1}t_n.&&\\[3mm]
\end{array}$$

Therefore, the first equality in \eqref{k2} for $1\leq s \leq n-2$ is not true for the chosen elements $x_i, y_i$ and $a\notin\{0,3\}$.

\item Let $s=n-1$ and $a\in \mathbb{N}$ such that $3a^2-a\neq 0$. Then $i=n$.

Consider the elements of $\cal A_{n}$:
$$x_j=t_{j+1}, \quad y_j=t_j, \quad 1\le j\le n-2, \quad y_{n-1}=t_{n}^a, \quad  y_{n}=t_1t_{n}.$$

It is not difficult to see that $Jac^q_{y_1, \dots, \widehat{y}_{p}, \dots, y_n}=0$ for $q\neq n-1$ and

$$Jac_{x_1,\dots,x_{n-2},y_p}^{l}=
\left\{\begin{array}{clll}
0,&\text{if}& l\notin\{1,n\},\\[1mm]
t_2t_3\dots t_{n-1}d_n(y_p),&\text{if}& l=1,\\[1mm]
(-1)^nt_2t_3\dots t_{n-1}d_1(y_p), & \text{if}& l=n.\\[1mm]
\end{array}\right.$$

%

%
%

Suppose $q=n-1$. Then the first equality of \eqref{k2} comes to
 $$   
 \sum\limits_{k=1}^{n-1}(-1)^{k-1}\Big(d_n(a_{k,k})-d_k(a_{n,k})\Big)=
(-1)^{n}(3\alpha^2-\alpha)(t_1t_2t_3\dots t_{n-2})^2t_{n-1}t_n^{\alpha+1}\neq0.$$

\end{itemize}
\end{exam}

From Example \ref{exam7.5} we derive that $d(\Omega^s(\alpha))$ is irreducible module over $S^n$ only for $s=0$.

Summarizing, with the results obtained above, we can state the next result.
\begin{thm} \label{thm8.3} The following $S_n$-modules:
\begin{itemize}
  \item $T(U, \alpha)$ with one-dimensional trivial $sl_{n}$-module $U=<u>$ under the action
$$[x_1, \dots, x_{n-1}, t^s\otimes u]=
\sum\limits_{i=1}^{n}(-1)^{n+i}
Jac^i_{x_1, \dots, x_{n-1}}d_i(t^s)\otimes u;$$

  \item $d(\Omega^0(\alpha))$
  under the action
  $$[x_1, \dots, x_{n-1}, \sum\limits_{i=1}^{n-1}d_i(t^s)\otimes e_i]=
\sum\limits_{i=1}^{n}d_i\Bigg(\sum_{k=1}^{n}(-1)^{n+}
Jac^k_{x_1, \dots, x_{n-1}}d_k(t^s)\Bigg)\otimes e_i,$$

where $\{e_1, \dots, e_{n-1}\}$ is a basis of the natural $sl_{n-1}$-module;

\item $1$-dimensional trivial,
\end{itemize}
are irreducible cuspidal $S^n$-modules.
\end{thm}
If $\alpha_i-\alpha'_i\in \mathbb{Z}^{n}$ for any $1\leq i \leq n$, then modules $T(U, \alpha)$ and $T(U, \alpha')$ are isomorphic. The same true for modules $d(\Omega^0(\alpha))$.The same we have for modules $d(\Omega^0(\alpha))$. Note that the regular module over $S^n$ corresponds to the module $T(\mathbb{C}, 0)$.

\

{\bf Acknowledgement} {{\small{\ The research problem studied in this paper was suggested to the first author by Efim Zelmanov during his stay at the University of California, San Diego. He is grateful to him for valuable comments and discussions throughout the completion of the results. B. Omirov thanks the SUSTech International Center for Mathematics, Shenzhen for the hospitality, where the main part of this work was carried out. He also acknowledges Vyacheslav Futorny for discussions on the representations of Lie algebras $W_n$ and $S_n$.}}
\\[3mm]
{\bf Data availability} {{\small{\ Data sharing not applicable to this article as no datasets were generated or analysed during
the current study.}}
\\[3mm]
{\bf Conflict of interest} {{\small{\ The authors have no competing interests to declare that are relevant to the content of this
article.}}
%

\end{document}